\documentclass[12pt]{article}

\usepackage[T1]{fontenc}
\usepackage[a4paper]{geometry}
\usepackage{subfigure}

\usepackage{algorithm}
\usepackage{algorithmicx}
\usepackage{algpseudocode}%
\usepackage{float}
\usepackage{amsmath}
\usepackage{amsthm}
\usepackage{amssymb}
\usepackage{mathptmx}
\usepackage{graphicx}
\usepackage{paralist}

\usepackage{todonotes}

\usepackage{enumitem}

\usepackage[colorlinks=true,linkcolor=blue,citecolor=blue]{hyperref}
\usepackage[capitalise]{cleveref}
\usepackage{soul}

\newcommand{\N}{\mathbb{N}}
\newcommand{\RR}{\mathbb{R}}
\newcommand{\iprods}[1]{\langle#1\rangle}
\newcommand{\Fc}{\mathcal{F}}
\newcommand{\Set}[1]{\left\{#1\right\}}
\newcommand{\dom}[1]{\mathrm{dom}\, #1}
\newcommand{\norms}[1]{\Vert#1\Vert}
\newcommand{\inte}[1]{\mathrm{int}\, #1}
\newcommand{\ri}[1]{\mathrm{ri}\, #1}
\newcommand{\iF}{\iota_{\Fc}}
\newcommand{\norm}[1]{\left\Vert#1\right\Vert}
\newcommand{\vertiii}[1]{{\left\vert\kern-0.25ex\left\vert\kern-0.25ex\left\vert #1
    \right\vert\kern-0.25ex\right\vert\kern-0.25ex\right\vert}}

\newtheorem{theorem}{Theorem}[section]

\theoremstyle{definition}

\newtheorem{lemma}{Lemma}[section]
\newtheorem{proposition}{Proposition}[section]

\theoremstyle{remark}
\newtheorem{remark}{Remark}[section]

\newtheorem{assumption}{Assumption}

\crefname{assumption}{Assumption}{Assumptions}
\Crefname{enumi}{}{}
\setlist[enumerate,1]{label=(\roman*)}

\algrenewcommand\algorithmicrequire{\textbf{Input:}}
\algrenewcommand\algorithmicensure{\textbf{Output:}}



\begin{document}

\title{The Boosted DC Algorithm for \\ linearly constrained DC programming}

\author{F. J. Arag\'on Artacho\footnote{Department of Mathematics, University of Alicante, Alicante, Spain. \newline\indent{~~} Email: francisco.aragon@ua.es}
\and
	R. Campoy\footnote{Department of Statistics and Operational Research, Universitat de Val\`{e}ncia, Valencia, Spain. \newline\indent{~~} Email: \texttt{ruben.campoy@uv.es}}
\and	
 P. T. Vuong\footnote{Mathematical Sciences School, University of Southampton, UK \newline\indent{~~} Email: \texttt{t.v.phan@soton.ac.uk}}
 }

\date{\today}

\maketitle

\begin{center}
Dedicated to Professor Miguel A. Goberna on the occasion of his 70th birthday
\end{center}

\begin{abstract}
	The Boosted Difference of Convex functions Algorithm (BDCA) has been recently introduced to accelerate the performance of the classical Difference of Convex functions Algorithm (DCA).
	This acceleration is achieved thanks to an extrapolation step from the point computed by DCA via a line search procedure.
	In this work, we propose an extension of BDCA that can be applied to difference of convex functions programs with linear constraints, and
    prove that every cluster point of the sequence generated by this algorithm is a Karush--Kuhn-Tucker point of the problem if the feasible set has a Slater point. When the objective function is quadratic, we prove that any sequence generated by the algorithm is bounded and R-linearly (geometrically) convergent.
	Finally, we present some numerical experiments where we compare the performance of DCA and BDCA on some challenging problems: to test the copositivity of a given matrix, to solve one-norm and infinity-norm trust-region subproblems, and to solve piecewise quadratic problems with box constraints. Our numerical results demonstrate that this new extension of BDCA outperforms DCA.
\end{abstract}


\section{Introduction}\label{sec:intro}
In this paper, we are interested in solving the following DC (difference of convex functions) optimization problem:
\begin{equation}\label{eq:P}\tag{$\mathcal{P}$}
\left\{\begin{array}{ll}
\displaystyle\min_{x\in\RR^n}& \phi(x) := g(x) - h(x) \vspace{1ex}\\
\text{s.t.} &\iprods{a_i, x}\leq b_i,~i=1,\ldots,p,
\end{array}\right.
\end{equation}
where $g:\RR^n\to\RR\cup\{+\infty\}$, $h:\RR^n\to\RR\cup\{+\infty\}$  are proper, closed, and convex functions with $g$ being smooth, $a_i\in\RR^n$, $b_i\in\RR$ for $i=1,\ldots,p$, and $\iprods{\cdot,\cdot}$ denotes an inner product. We use the conventions:
\begin{gather*}
(+\infty)-(+\infty)=+\infty,\\
(+\infty)-\lambda=+\infty\quad
\text{and}\quad\lambda-(+\infty)=-\infty,\quad\forall\lambda\in{]-\infty,+\infty[}.
\end{gather*}

Observe that we can rewrite problem~\eqref{eq:P} as an unconstrained nonsmooth DC optimization problem, whose objective function is $g + \iota_{\Fc} - h$, where $\iota_{\Fc}$ denotes the indicator function of the feasible set
$$\Fc := \Set{x\in\RR^n \mid \iprods{a_i, x}\leq b_i,\, i=1,\ldots,p}.$$
For solving this problem, one can apply the classical DC Algorithm (DCA) in \cite{TT97,An2018}.
DC~programming and the DCA have been developed and studied for more than 30 years~\cite{An2018}.
The DCA has been successfully applied in different fields such as machine learning, financial optimization, supply chain management, and telecommunication, see, e.g. \cite{tao2005dc,Geremew2018, Nam2018}.
Nowadays, DCA has become a useful method to solve nonconvex problems.

To accelerate the convergence of DCA, which can be slow for some problems, a new method called \textit{Boosted DC Algorithm} (BDCA) has been recently proposed in \cite{BDCA2018,nBDCA}.
The key idea of BDCA is to perform an extrapolation step via a line search procedure at the point computed by DCA at each iteration.
This step allows the algorithm to take longer steps than the classical DCA, achieving in this way a larger reduction of the objective value per iteration.
In addition to accelerating its convergence, BDCA may have better chances to escape from bad local optima thanks to the line search procedure, see~\cite[Example~3.3]{nBDCA}.
Therefore, BDCA is not only faster than DCA but also can provide better solutions.
Extensive numerical experiments in diverse applications such as  biochemistry \cite{BDCA2018}, machine learning~\cite{XuXue2017,ML2} and data science \cite{nBDCA} have been performed where BDCA clearly outperforms DCA. Note that all these applications are modeled as DC problems with $g$ smooth.
However, it is important to emphasize that, for unconstrained DC programs, the BDCA proposed in~\cite{BDCA2018,nBDCA} is not applicable when the function $g$ in \eqref{eq:P} is nonsmooth (see~\cite[Example~3.4]{nBDCA}).

The aim of this paper is to show that BDCA can still be applied if the nonsmooth function $g$ is the sum of a smooth convex function and the indicator function of a polyhedral set.
More precisely, we will show that  it is possible to use BDCA for solving DC programs with linear constraints of the form \eqref{eq:P}.
To compare the performance of DCA and BDCA, we provide numerical experiments on various challenging problems: to test the copositivity of a given matrix, to solve $\ell_1$ and $\ell_\infty$ trust-region subproblems, and to find the minimum of a piecewise quadratic function with box constraints.
The three first problems have a quadratic objective function and are known to be NP-hard~\cite{Murty} (the first was already heuristically investigated in \cite{Durr} using DCA), while the last problem has a nonsmooth objective function. Our results confirm that BDCA significantly outperforms DCA in these applications.

The rest of this paper is organized as follows.
\Cref{sec:pre_results} recalls some preliminary results.
In Section~\ref{sec:BDCA}, we propose a new variant of BDCA for solving \eqref{eq:P} and prove that the objective value of the sequence generated by the algorithm is monotonically decreasing, and that any limit point of the sequence is a KKT point of the problem. The R-linear convergence of any sequence generated by BDCA in the special case of quadratic objective functions is derived in \Cref{sec:linearRate}. In \Cref{sec:application1}, we provide some numerical experiments for testing the copositivity of a given matrix, for solving $\ell_1$ and $\ell_\infty$ trust-region subproblems and {for solving piecewise quadratic} problems with box constraints, where we compare BDCA and DCA. Finally, some conclusions and future research are briefly discussed in~\Cref{sec:conclusions}.

\section{Preliminaries}\label{sec:pre_results}
In this section, we state our assumptions imposed on \eqref{eq:P}.
We also recall some preliminary and basic results which will be used in the sequel.

Let {$f:\RR^n\to\RR\cup\Set{+\infty}$} be a proper extended real-valued convex function. The set $\dom{f} := \Set{x\in\RR^n \mid f(x) < +\infty}$ denotes its (effective) \emph{domain}, and $$\partial{f}(x) := \Set{ w\in\RR^n \mid f(y) \geq f(x) + \iprods{w,y-x},~\forall y\in\RR^n}$$ denotes the \emph{subdifferential} of $f$ at $x$. If $f$ is differentiable at $x$, then $\partial{f}(x)=\{\nabla f(x)\}$, where $\nabla f(x)$ stands for the \emph{gradient} of $f$ at $x$. The one-side \emph{directional derivative} of $f$ at $x$ with respect to the direction $d\in\RR^n$ is
$$f'(x;d):=\lim_{t\searrow 0}\frac{f(x+td)-f(x)}{t}.$$
Recall that $f$ is said to be \emph{strongly convex} with \emph{strong convexity parameter} $\rho>0$ if the function $f-\frac{\rho}{2}\norms{\cdot}^2$ is convex. The \emph{conjugate} of $f$ is the function $f^\ast:\RR^n\to\RR\cup\Set{+\infty}$ defined at $u\in\RR^n$ as
$$f^\ast(u)=\sup_{x\in\RR^n}\{\langle x,u\rangle -f(x)\}.$$
It holds that $f^{\ast\ast}:=(f^\ast)^\ast=f$. Moreover, for all $x,u\in\RR^n$,
$$u\in\partial f(x) \quad\Leftrightarrow\quad f(x)+f^\ast(u)=\langle x, u\rangle \quad\Leftrightarrow\quad x\in\partial f^\ast(u).$$
Given a set $C\subseteq\RR^n$, its \emph{interior}, its relative interior and the \emph{convex cone generated} by $C$ are denoted by $\inte C$, $\ri C$ and $ \mathop{\rm cone} C$, respectively. The function
$$\iota_C(x):=\left\{\begin{array}{ll}
0, &\text{if } x\in C,\\
+\infty, & \text{otherwise;}
\end{array}\right.$$
denotes the \emph{indicator function} of $C$, and $C$ is convex if and only if $\iota_C$ is so. Moreover, its subdifferential $\partial\iota_C$ is the \emph{normal cone} to $C$, which is given by
$$N_C(x):=\left\{\begin{array}{ll}
\{u\in\RR^n : \langle u, c-x \rangle\leq 0, \quad \forall c\in C \}, &\text{if }x\in C,\\
\emptyset, & \text{otherwise.}
\end{array}\right.$$

For proving our convergence results, we will make use of the following assumptions.

\begin{assumption}\label{as:A1}
Both $g$ and $h$ are strongly convex on their domain with the same strong convexity parameter $\rho >0$.
\end{assumption}

\begin{assumption}\label{as:A2}
	The function $h$ is subdifferentiable at every point in $\dom{h}$; i.e., $\partial h(x)\neq\emptyset$ for all $x\in\dom{h}$.
	The function $g$ is continuously differentiable on an open set containing $\dom{h}$  and
	\begin{equation}\label{eq:phi_bounded_below}
	\inf_{x\in\Fc}\phi(x) > -\infty.
	\end{equation}
\end{assumption}

\begin{assumption}\label{as:A3}
The feasible set $\mathcal{F}$ has a Slater point, that is, there exists $\hat{x}\in \RR^n$ such that $\langle a_i,\hat{x}\rangle <b_i$, for all $i=1,\ldots,p$.
\end{assumption}

\begin{remark}
Assumption~\ref{as:A1} is not restrictive, in the sense that any DC decomposition of $\phi$ as $\phi = g - h$, can be expressed as $\phi = (g+\frac{\rho}{2}\norms{\cdot}^2) - (h+\frac{\rho}{2}\norms{\cdot}^2)$ for any $\rho>0$.
Observe that $\partial h(x)\neq\emptyset$ holds for all $x\in\ri\dom{h}$ (by \cite[Theorem~23.4]{RT}), so the first part of \Cref{as:A2} is clearly satisfied if $\dom{h}=\RR^n$. A key point of our method is the smoothness of $g$ in Assumption~\ref{as:A2}, which cannot be in general omitted (see~\cite[Example~3.4]{nBDCA}).
\end{remark}

Associated with problem \eqref{eq:P}, we can construct its \emph{dual} as
\begin{equation}\label{eq:D}\tag{$\mathcal{D}$}
\displaystyle\inf_{u\in\RR^n}\, h^\ast(u)-(g+\iF)^\ast(u),
\end{equation}
which is also a DC program with the same optimal value. Indeed,
\begin{align*}
\inf_{u\in\RR^n}\{h^\ast(u)-(g+\iF)^\ast(u)\}&=\inf_{u\in\RR^n}\{h^\ast(u)-\sup_{x\in\RR^n}\{\langle u,x\rangle - g(x)-\iF(x)\}\}\\
&=\inf_{x\in\RR^n}\{g(x)+\iota_F(x)-\sup_{u\in\RR^n}\{\langle u, x\rangle -h^\ast(u)\}\}\\
&=\inf_{x\in\RR^n}\{g(x)+\iota_F(x)-h(x)\}\\
&=\inf_{x\in\Fc}\{g(x)-h(x)\}=\inf_{x\in\Fc}\phi(x).
\end{align*}
We say that $\overline{x}\in \RR^n$ (resp. $\overline{u}\in \RR^n$) is a \emph{critical point }of~\eqref{eq:P} (resp.~\eqref{eq:D}) if $\nabla g(\overline{x})\in\partial (h+\iF)(\overline{x})$ (resp. $\partial h^\ast(\overline{u})\cap\partial (g+\iF)^\ast(\overline{u})\neq\emptyset$). Recall from \cite[Theorem 5.19]{Mord} that $\overline{x}$ is called a \emph{KKT point} of~\eqref{eq:P} if there exist $\mu_1,\mu_2,\ldots,\mu_p \in \RR$ such that
\begin{equation}\label{KKT}
\left\{\begin{array}{l}
0 \in \nabla{g}(\overline{x}) -\partial{h}(\overline{x}) +\sum_{i=1}^p\mu_i a_i, \vspace{1ex}\\
0 = \mu_i (\iprods{a_i,\overline{x}}-b_i),~~i=1,\ldots,p,\vspace{1ex}\\
\mu_i\geq 0, ~~\iprods{a_i,\overline{x}}\leq b_i,~~i=1,\ldots,p.
\end{array}\right.
\end{equation}
The DC algorithm is a primal-dual method in the sense that it is aimed to find solutions for both \eqref{eq:P} and \eqref{eq:D}, simultaneously. Our goal then is to design a BDCA variant that allow us to find KKT points of~\eqref{eq:P} and critical points of~\eqref{eq:D}.

To finish this section, we need to introduce the following notions regarding the geometry of the feasible set $\Fc$. The  \emph{cone of feasible directions} at $\overline{x}\in \Fc$ is denoted by
\begin{equation*}
D( \overline{x}) :=\Set{ d\in \RR^{n} \mid \exists
\varepsilon >0~\text{such that}~\overline{x}+td\in \Fc,~\forall t\in \left[ 0,\varepsilon \right] }.
\end{equation*}
Due to the polyhedral structure of $\Fc$, its normal cone coincides with the \emph{active cone}; i.e.,
\begin{equation*}
N_\Fc(\overline{x})=A(\overline{x}) := \mathop{\rm cone}\Set{ a_i, ~~i\in I(\overline{x})}, \quad \text{for } \overline{x}\in\Fc,
\end{equation*}
where $I(\overline{x})$ stands for the set of \emph{active constraints} at $\overline{x}$, i.e., $$I(\overline{x})=\{i\in\{1,\ldots,p\}\mid \langle a_i,\overline{x}\rangle = b_i\}.$$
Since we deal with affine constraints, we have (see e.g. \cite[Proposition~4.14]{book})
\begin{equation}\label{eq:feas_dir}
D(\overline{x})=\Set{ d \in\RR^n \mid \iprods{a_i, d} \leq 0,~ i\in I(\overline{x})}.
\end{equation}

\section{The Boosted DC Algorithm and its convergence}\label{sec:BDCA}
For solving \eqref{eq:P}, we propose the following method, \Cref{alg:BDCA}, which can tackle more general problems than the Boosted DC Algorithm proposed in \cite{nBDCA}.\\

\begin{algorithm}[H]\caption{BDCA (Boosted DC Algorithm) for solving \eqref{eq:P}}\label{alg:BDCA}
\begin{algorithmic}[1]
\Require{An initial point $x_{0}\in \Fc$ and two parameters $\alpha>0$ and $\beta\in{]0,1[}$;}
\State $k\leftarrow 0$;
\State Select $u_k \in \partial h(x_k)$ and compute the unique solution $y_k$ of
\begin{equation}\label{eq:sub_prob}
\tag{$\mathcal{P}_{k}$} ~
\left\{\begin{array}{ll}
\displaystyle\min_{x\in\RR^n} &\phi_k(x) := g(x) - \iprods{u_k, x} \vspace{1ex}\\
\textrm{s.t.}& \iprods{a_i, x} \leq b_i,~i=1,\ldots,p.
\end{array}\right.
\end{equation}%
\State $d_{k}\leftarrow y_{k}-x_{k}$;

\If{$d_{k}=0$}
\State \textbf{stop} and \textbf{return} $x_{k}$;
\EndIf

\If{$I(y_k)\subseteq I(x_k)$}
\State Choose any $\overline{\lambda}_k\geq 0$, set $\lambda_k\leftarrow\overline{\lambda}_k$, and reduce $\lambda_k$ until $y_{k}+\lambda_{k}d_{k}\in \Fc$;
\While{$\phi(y_{k}+\lambda_{k}d_{k})>\phi(y_{k})-\alpha\lambda_{k}^2\|d_{k}\|^{2}$}
\State $\lambda_{k}\leftarrow\beta\lambda_{k}$;
\EndWhile
\Else
\State $\lambda_k\leftarrow 0$;
\EndIf
\State $x_{k+1}\leftarrow y_{k}+\lambda_{k}d_{k}$;
\State $k\leftarrow k+1$ and \textbf{go to} Line~2;
\end{algorithmic}
\end{algorithm}\bigskip

Let us make some comments on \Cref{alg:BDCA}.
\begin{itemize}
\item[(i)] Lines 1 to 6 of \Cref{alg:BDCA} correspond to the classical DCA for solving \eqref{eq:P}. Thus, when $\overline{\lambda}_k=0$ for all $k$, \Cref{alg:BDCA} coincides with DCA.
\item[(ii)] Lines 7 to 15 present the boosting step. It first checks if $d_k$ is a feasible direction at $y_k\in\Fc$. If so, it then performs a line search step along the direction $d_k$ which maintains feasibility to improve the objective value $\phi$. Otherwise, the boosting step is skipped and  we simply use the DCA~point~$y_k$.

\item[(iii)] In terms of per-iteration complexity, the boosting step requires to check the feasibility of direction $d_k$, which can be done by comparing the sets of active constraints at $x_k$ and $y_k$ (see~\Cref{lemma:feas_dir}).
It also requires evaluating the objective function and checking the feasibility of the trial step $y_k + \lambda_kd_k$.
The computational effort of this task will depend on the particular structure of $\phi$ and $\Fc$. In the case of box constraints, the largest step-size that makes $y_k + \lambda_kd_k$ feasible can be efficiently computed in Line~8, see \Cref{rem:maxlambda} below. In Section~\ref{sec:application1} we show some computational results for the trust-region subproblems with $\ell_1$ norm constraints (where feasibility needs to be checked) and $\ell_\infty$ norm constraints (where the largest step-size can be readily computed).

\item[(iv)] When $h$ is differentiable, the algorithms introduced by Fukushima and Mine in~\cite{FM81,MF81} can be applied in our setting. On the one hand, the algorithm in~\cite{FM81} performs a line search which is similar to the one in Lines~9-11 of \Cref{alg:BDCA}, but with the significant difference that it is performed at the point $x_k$, instead of doing it at the DCA point $y_k$; that is, it searches for the smallest non-negative integer $l$ such that
$$\phi\left(x_k+\beta^l d_k\right)\leq \phi(x_k)-\alpha\beta ^l\|d_k\|^2,$$
where $0<\beta<1$. Thus, the largest step-size allowed by their algorithm is $1$, which corresponds with the point determined by the DCA, since $x_k+d_k=y_k$. On the other hand, the algorithm defined in~\cite{MF81} performs an exact line search in the direction $d_k$, which may be unaffordable in many practical applications.
\end{itemize}

\begin{remark}\label{rem:maxlambda}
An explicit formula to compute the largest step-size that makes $y_k + \lambda_kd_k$ feasible in Line~8 is provided in \cite{FSS21}. More precisely, feasibility of the trial step size is guaranteed if $\overline\lambda_k$ is chosen so that
\begin{equation}\label{lambda_bound}
\overline\lambda_k\leq\widehat\lambda_k:=
\min\Set{\frac{b_i-\langle a_i,y_k\rangle}{\lvert\langle a_i,d_k\rangle\rvert}: i\not\in I(x_k) \text{ with } \langle a_i,d_k\rangle\neq 0 };
\end{equation}
see \cite[Lemma~5.4(ii)]{FSS21}. In principle, this permits to avoid the extra time consumed for checking feasibility in Line 8. However, the computation of $\widehat\lambda_k$ in \eqref{lambda_bound} may be even more expensive and inefficient  {than checking feasibility of $\lambda_k$} in some applications. This is the case, for instance, of the $\ell_1$-norm trust region subproblem, whose reformulation as a linearly constrained DC program requires an exponential number of constraints (see \Cref{sec:trsp}).
\end{remark}

The next auxiliary lemma shows the equivalence between Line~7 of \Cref{alg:BDCA} and checking the feasibility of the direction generated by DCA.
\begin{lemma}\label{lemma:feas_dir}
If $x_k$ and $y_k$ are generated by~\Cref{alg:BDCA}, then
$$I(y_k)\subseteq I(x_k) \quad\Leftrightarrow\quad d_k:=y_k-x_k\in D(y_k)\quad\Leftrightarrow\quad d_k\,\bot\, a_i,\;\forall i\in I(y_k).$$
\end{lemma}
\begin{proof}
Observe that, for any $i\in I(y_k)$, it holds that
\begin{align*}
\iprods{ a_i, d_k } & = \iprods{a_i, y_k}-\iprods{a_i,x_k} =b_i - \iprods{a_i,x_k}\geq 0.
\end{align*}
Hence, the result easily follows by taking into account~\eqref{eq:feas_dir}.
\end{proof}

In the following proposition, we collect some key inequalities which are useful in the sequel for the convergence analysis of \Cref{alg:BDCA}.

\begin{proposition}\label{prop:main_inequality}
Under Assumptions~\ref{as:A1} and \ref{as:A2}, for all $k\in\N$, the next statements hold:
\begin{enumerate}
\item $\phi(y_{k})\leq\phi(x_{k})-\rho\|d_{k}\|^{2}$;\label{prop:main_inequalityI}
\item $\phi'(y_k;d_k) \leq-\rho\|d_{k}\|^{2}$;\label{prop:main_inequalityII}
\item if the condition at Line~7 of \Cref{alg:BDCA} holds, then there exists some $\delta_k > 0$ such that $y_{k}+\lambda_{k}d_{k}\in \Fc$ and\label{prop:main_inequalityIII}
\begin{equation*}
\phi\left(y_{k}+\lambda d_{k}\right)\leq\phi(y_{k})- \alpha \lambda^2 \|d_{k}\|^{2},\quad\text{for all }\lambda\in[0,\delta_k].\label{eq:main_inequality_bdca}
\end{equation*}
Consequently,  the backtracking step at Lines 9--11 of \Cref{alg:BDCA} terminates after a finite number of iterations.
\end{enumerate}
\end{proposition}
\begin{proof}
The proof of~\Cref{prop:main_inequalityI} is similar to the one of~\cite[Proposition~3]{BDCA2018} and is therefore omitted.
To prove~\Cref{prop:main_inequalityII}, pick any $v\in\partial h(y_k)$. Note that the one-sided directional derivative $\phi'(y_k;d_k)$ is
given by
\begin{align}\label{directionalDerivative}
\phi'(y_k;d_k) \nonumber
&= \lim_{t \downarrow 0} \frac{\phi(y_k+ t d_k)-\phi(y_k)}{t}\\\nonumber
&= \lim_{t \downarrow 0} \frac{g(y_k+ t d_k)-g(y_k)}{t} - \lim_{t \downarrow 0} \frac{h(y_k+ t d_k)-h(y_k)}{t}\\
&\leq \left\langle \nabla g(y_k), d_k \right\rangle  - \left\langle v, d_k \right\rangle,
\end{align}
by convexity of $h$.
Since $y_{k}$ is the unique solution of the strongly convex
problem \eqref{eq:sub_prob}, we can write down the KKT conditions (see, e.g., \cite[Theorem~4.20]{book}) of this problem as
\begin{equation}\label{KKTyk}
\left\{\begin{array}{l}
\nabla{g}(y_{k})+\sum_{i=1}^p\mu_{k,i} a_i=u_k\in\partial h(x_k), \vspace{1ex}\\
\mu_{k,i} (\iprods{a_i,y_k}-b_i)=0, ~~\mu_{k,i}\geq 0,~~ \iprods{a_i,y_k}\leq b_i,~~i=1,\ldots,p.
\end{array}\right.
\end{equation}
The fact that $h$ is strongly convex with a parameter $\rho$ implies, by \cite[Exercise 12.59]{RW}, that $\partial h$ is strongly monotone with constant $\rho$.
Therefore, since $v \in \partial h(y_k)$ and $u_k\in\partial h(x_k)$, we have
\begin{equation*}
\iprods{u_k-v,x_{k}-y_{k}} \geq\rho\|x_{k}-y_{k}\|^{\text{2}}.
\end{equation*}
Hence, combining these expressions, together with the fact that $x_k\in\Fc$, we can derive
\begin{align*}
\iprods{\nabla g(y_{k})-v,d_{k}} &= \Big\langle u_k-\sum_{i=1}^p\mu_{k,i} a_i-v,y_{k}-x_k\Big\rangle\\
&\leq-\rho\|d_{k}\|^2-\sum_{i=1}^p\mu_{k,i}\langle a_i,y_k-x_k\rangle\\
&=-\rho\|d_{k}\|^2+\sum_{i=1}^p\mu_{k,i}\left(\iprods{a_i,x_k}-b_i\right)+\sum_{i=1}^p \mu_{k,i}\left(b_i-\iprods{a_i, y_k}\right)\\
&\leq-\rho\|d_{k}\|^2,
\end{align*}
and the result follows by combining the last inequality with \eqref{directionalDerivative}.

Having in mind the condition at Line~7 of \Cref{alg:BDCA} and \Cref{lemma:feas_dir}, we observe that the proof of~\Cref{prop:main_inequalityIII} is similar to the one of~\cite[Proposition~3.1]{nBDCA}, so we omit it for brevity.
\end{proof}

\begin{remark}[General convex constraints]\label{re:general_convex}
Consider a generalized version of \eqref{eq:P} where the feasible set is formed by arbitrary convex constraints, i.e.,
\begin{equation}\label{eq:PG}\tag{$\mathcal{P}'$}
\left\{\begin{array}{ll}
\displaystyle\min_{x\in\RR^n}& \phi(x) := g(x) - h(x) \vspace{1ex}\\
\text{s.t.} &c_i(x)\leq 0,~i=1,\ldots,p,
\end{array}\right.
\end{equation}
where $g$ and $h$ satisfy \Cref{as:A1,as:A2} and $c_i:\RR^n\to\RR$ are smooth, proper, closed and convex functions, for $i=1,\ldots,p.$ Note that problem \eqref{eq:P} is a particular instance of \eqref{eq:PG} with $c_i(x):=\iprods{a_i,x}-b_i$, for $i=1,\ldots,p.$ The assertion in \Cref{prop:main_inequality}\Cref{prop:main_inequalityII} still holds true for the more general problem~\eqref{eq:PG}; that is, the direction generated by DCA remains a descent direction provided that $x_k$ is feasible for \eqref{eq:PG}. To confirm this, one can easily check that the proof can be rewritten by replacing the linearity of the gradients by the inequality
\begin{equation}\label{convci}
c_i(x_k)\geq c_i(y_k)-\iprods{\nabla{c}_i(y_k),y_k-x_k}.
\end{equation}
However, Line~7 of~\Cref{alg:BDCA} is no longer useful to verify if $d_k$ is a feasible direction, as the equality in~\eqref{eq:feas_dir} only holds for affine constraints. For general convex constraints, we have the inclusion
\begin{equation*}
\Set{d\in\RR^n \mid \iprods{\nabla{c_i}(\overline{x}), d} < 0,~i\in I(\overline{x})} \subset D(\overline{x}).
\end{equation*}
Therefore, one possibility would be to run the boosting step whenever $\langle \nabla c_i(y_k),d_k\rangle <0$ for all $i\in I(y_k)$. Nevertheless, this will never be the case because $x_k$ is feasible for \eqref{eq:PG}. Indeed, from~\eqref{convci}, we obtain that
\begin{equation*}
\iprods{\nabla{c}_i(y_k),d_k}\geq -c_i(x_k)\geq 0,\quad\text{for all } i\in I(y_k).
\end{equation*}
In fact, it can be proved that if $y_k+\lambda d_k\in\Fc$ for some particular $\lambda>0$, then the points in the segment $[x_k, y_k+\lambda d_k]$ must be active for all $i\in I(y_k)$.
\end{remark}

We are now in the position to establish the main convergence result of \Cref{alg:BDCA}.

\begin{theorem}\label{prop:innerloop}
Under Assumptions~\ref{as:A1}, \ref{as:A2} and \ref{as:A3}, for any $x_{0}\in \Fc$, either BDCA returns a KKT point of~\eqref{eq:P} or it generates an infinite sequence such that the following statements hold.
\begin{enumerate}
\item $\phi(x_{k})$ is monotonically decreasing and hence convergent to some $\bar{\phi}$.\label{prop:innerloopI}

\item Suppose that $\{x_k\}$ and $\{u_k\}$ are bounded. Then any limit point of $\Set{x_{k}}$ is a KKT point of \eqref{eq:P} and any limit point of $\{u_{k}\}$ is a critical point of \eqref{eq:D}.
\label{prop:innerloopII}

\item We have $\sum_{k=0}^{+\infty}\norms{d_{k}}^{2}<+\infty$.
Moreover, if there is some $\overline{\lambda}$ such that $\lambda_k \leq \overline{\lambda}$ for all $k$, then $\sum_{k=0}^{+\infty}\norms{x_{k+1}-x_{k}}^{2}<+\infty$.\label{prop:innerloopIII}
\end{enumerate}
\end{theorem}

\begin{proof}
If \Cref{alg:BDCA} is terminated at Line 5 and returns $x_k$, then $x_k = y_k$.
From \eqref{KKT} and~\eqref{KKTyk}, it is clear that $x_k$ is a KKT point of \eqref{eq:P}.
Otherwise, by Proposition \ref{prop:main_inequality} and Line 15 of \Cref{alg:BDCA}, we have
\begin{equation}\label{eq:phi_decreasing}
\phi(x_{k+1})\leq\phi(y_{k})-\alpha\lambda^2_{k}\|d_{k}\|^{2}\leq\phi(x_{k})-\left(\alpha\lambda^2_{k}+\rho\right)\|d_{k}\|^{2},
\end{equation}
where $\lambda_k \geq 0$.
Therefore, the sequence $\Set{ \phi(x_{k})}$ converges to some $\bar{\phi}$, since it is monotonically decreasing
and bounded from below, by \eqref{eq:phi_bounded_below}.
As a consequence, we obtain
\begin{equation*}
\phi(x_{k+1})-\phi(x_{k})\to 0,~\text{as}~k\to\infty,
\end{equation*}
which implies $\norms{d_{k}}^{2}=\norms{y_{k}-x_{k}}^{2}\to 0$, by~\eqref{eq:phi_decreasing}.

Now, if $\overline{x}$ is a limit point of $\Set{x_{k}}$, then there exists a subsequence $\Set{x_{k_{j}}}$ converging to~$\overline{x}$.
Then, as $\norms{y_{k_{j}}-x_{k_{j}}}\to 0$, we have $y_{k_{j}}\to\overline{x}$.
From \eqref{KKTyk}, we obtain
\begin{equation} \label{KKTbis}
\left\{\begin{array}{l}
\nabla{g}(y_{k_j}) + \sum_{i=1}^p\mu_{k_j,i} a_i =u_{k_j}\in \partial h(x_{k_j}),\vspace{1ex}\\
\mu_{k_j,i} (\iprods{a_i,y_{k_j}}-b_i)=0, ~~\mu_{k_j,i}\geq 0,~~ \iprods{a_i,y_{k_j}}\leq b_i,~~i=1,\ldots,p.
\end{array}\right.
\end{equation}
Since the sequence $\{u_k\}$ is bounded by assumption, without lost of generality we may assume that $u_{k_j}\to \overline{u}$. It also holds that $\nabla{g}(y_{k_j})\to \nabla{g}(\overline{x})$ by continuity of $\nabla g$. The sequence of Lagrange multipliers $\{\mu_{k_j}\}\in\mathbb{R}^p$ must also be bounded. Indeed, suppose to the contrary that $\|\mu_{k_j}\|\to\infty$ and assume without loss of generality that $\lim_{j\to\infty}\frac{\mu_{k_j}}{\|\mu_{k_j}\|}=\mu^*$, with $\mu^*\in\RR^p_+$ (the non-negative orthant) satisfying $\|\mu^*\|=1$. Then, dividing the first equality in~\eqref{KKTbis} by $\|\mu_{k_j}\|$ and letting $j\to\infty$, we obtain
$$\sum_{i=1}^p \mu_i^*a_i=0_n.$$
Performing the same procedure in the second equality in~\eqref{KKTbis} gives
$$\mu_i^*(\langle a_i,\overline{x}\rangle-b_i)=0,$$
so we deduce that $\mu_i^*=0$ for all $i\not\in I(\overline{x})$. Thus,
$$\sum_{i\in I(\overline{x})} \mu_i^*a_i=0_n.$$
Thanks to the Slater Assumption~\ref{as:A3}, we know that there exists $\hat{x}\in\RR^n$ such that $\langle a_i,\hat{x}\rangle<b_i$ for all $i\in I(\overline{x})$. Hence,
$$0=\sum_{i\in I(\overline{x})} \mu_i^*\langle a_i,\hat{x}-\overline{x}\rangle=\sum_{i\in I(\overline{x})} \mu_i^*\left(\langle a_i,\hat{x}\rangle-\langle a_i,\overline{x}\rangle\right)=\sum_{i\in I(\overline{x})} \mu_i^*\left(\langle a_i,\hat{x}\rangle-b_i\right),$$
which implies $\mu_i^*=0$ for all $i\in I(\overline{x})$, since $\mu^*\in\RR^p_+$. This implies that $\mu^*=0_p$, a contradiction with the fact that $\|\mu^*\|=1$.

Therefore, by extracting subsequences if necessary, we can assume that
\begin{equation}\label{lim_mus}
\lim_{j\to\infty}\mu_{k_j,i} =\mu_i\geq 0,\quad \text{for all } i=1,2,\ldots,p.
\end{equation}%
Taking the limit as $j \to \infty$ in \eqref{KKTbis}, thanks to the closedness of the graph of $\partial{h}$ (see~\cite[Theorem 24.4]{RT}), we obtain
\begin{equation}\label{KKTlimit}
\left\{\begin{array}{l}
\nabla{g}(\overline{x}) +\sum_{i=1}^p\mu_i a_i = \overline{u}\in \partial h(\overline{x}),\vspace{1ex}\\
\mu_i (\iprods{a_i,\overline{x}}-b_i)=0, ~~\mu_i\geq 0,~~ \iprods{a_i,\overline{x}}\leq b_i,~~i=1,\ldots,p,
\end{array}\right.
\end{equation}
which means that $\overline{x}$ is a KKT point of \eqref{eq:P}. From \eqref{KKTlimit} we derive that $\overline{x}\in\partial h^\ast(\overline{u})$ and also that
\begin{equation*}
\overline{u}=\nabla g(\overline{x})+ \sum_{i\in I(\overline{x})}\mu_ia_i \in \nabla g(\overline{x})+A(\overline{x})=\partial (g+\iF)(\overline{x}),
\end{equation*}
which is equivalent to $\overline{x}\in\partial (g+\iF)^\ast(\overline{u})$ and, hence, $\overline{u}$ is a critical point of \eqref{eq:D}.
The proof of \Cref{prop:innerloopIII} is similar to that of \cite[Proposition 5(iii)]{BDCA2018} and is thus omitted.
\end{proof}

\begin{remark}\label{rem:int}
The assertion in \Cref{prop:innerloop}\Cref{prop:innerloopII} remains valid for any limit point of $\{x_k\}$ in the interior of $\dom h$ without requiring the dual sequence $\{u_k\}$ to be bounded.
Indeed, let $\overline{x}\in\inte\dom(h)$ be a cluster point of $\{x_k\}$ and let $\{x_{k_j}\}$ be a subsequence converging to $\overline{x}$. We can apply~\cite[Corollary~24.5.1]{RT} to obtain that for any $\varepsilon>0$, there exists a positive integer $j_0\geq 1$ such that
\begin{equation*}
u_{k_j}\in\partial h(x_{k_j})\subset \partial h(\overline{x})+\varepsilon\mathbb{B}, \quad \forall j\geq j_0;
\end{equation*}
where $\mathbb{B}$ denotes the Euclidean unit ball of $\RR^n$. Hence, since $\partial h(\overline{x})$ is a bounded set, the dual subsequence $\{u_{k_j}\}$ is also bounded and the proof of \Cref{prop:innerloop}\Cref{prop:innerloopII} stands.
\end{remark}

\begin{remark}
Similar to \cite[Theorem 1]{BDCA2018} and \cite[Theorem 4.3 and Theorem 4.9]{nBDCA}, if we further assume that the function $\phi$ satisfies the Kurdyka--\L{}ojasiewicz property, then it can be proved that the sequence $\Set{x_k}$ converges to a KKT point of \eqref{eq:P}.
Moreover, convergence rates can also be deduced depending on the \L{}ojasiewicz exponent.
Especially, when the objective function $\phi$ is quadratic (e.g., in some of our numerical experiments), it was proved \cite[Theorem 4.2]{JOTA2018} that the function $\phi+\iota_{\Fc}$ satisfies the  Kurdyka--\L{}ojasiewicz property with exponent $\frac{1}{2}$. Combining this with the technique in \cite[Theorem 1]{BDCA2018}, it is a routine task to derive the linear convergence of  the sequence $\Set{x_k}$ when the sequence has a cluster point. The purpose of the next section is to prove, using a similar technique to \cite[Theorem~2.1]{TuanJMAA}, that the latter condition is not needed: for quadratic functions, BDCA always generates a sequence which is linearly convergent to a KKT point of the problem without requiring the existence of a cluster point.
\end{remark}

\section{Linear convergence for quadratic objective functions}\label{sec:linearRate}
In this section we prove the R-linear convergence of BDCA when the objective function $\phi$ of \eqref{eq:P}  is quadratic, that is, for problems of the form
\begin{equation}\label{eq:Pq}\tag{$\mathcal{P}_Q$}
\left\{\begin{array}{ll}
\displaystyle\min_{x\in\RR^n}& \phi(x) := \frac{1}{2}\langle Qx,x\rangle+ \langle q,x\rangle \vspace{1ex}\\
\text{s.t.} &\iprods{a_i, x}\leq b_i,~i=1,\ldots,p,
\end{array}\right.
\end{equation}
where $Q\in\RR^{n\times n}$ { is a symmetric matrix,} $q\in\RR^{n}$, and $a_i\in\RR^n$ and $b_i\in\RR$, for $i=1,\ldots,p$.
In this setting, $\overline{x}$ is a KKT point of \eqref{eq:Pq} if there exist multipliers $\mu_1,\mu_2,\ldots,\mu_p \in \mathbb{R}$ such that
\begin{equation*}\label{KKT_Pq}
\left\{\begin{array}{l}
0=Q\overline{x}+q +\sum_{i=1}^p\mu_i a_i, \vspace{1ex}\\
0 = \mu_i (\iprods{a_i,\overline{x}}-b_i),~~i=1,\ldots,p,\vspace{1ex}\\
\mu_i\geq 0, ~~\iprods{a_i,\overline{x}}\leq b_i,~~i=1,\ldots,p.
\end{array}\right.
\end{equation*}
We denote by $\overline{\Fc}$ the set of all KKT points of \eqref{eq:P}.

As $Q$ is not required to be positive semidefinite, \eqref{eq:Pq} is a nonconvex problem. However, the matrix $Q$ can be easily decomposed as $Q=Q_1-Q_2$, with $Q_1$ and $Q_2$ positive definite. Indeed, one can take $$Q_1:=\sigma I \quad\text{and}\quad Q_2:=\sigma I-Q,$$ for $\sigma > \max\{0,\lambda_{\max}(Q)\}$, where $\lambda_{\max}(Q)$ is the largest eigenvalue of $Q$ and $I$ denotes the identity matrix. Thus, \eqref{eq:Pq} can be equivalently written in the form of \eqref{eq:P} as
\begin{equation}\label{eq:Pq_decomposition}
\left\{\begin{array}{ll}
\displaystyle\min_{x\in\RR^n}& g(x) - h(x)=\phi(x) \vspace{1ex}\\
\text{s.t.} &\iprods{a_i, x}\leq b_i,~i=1,\ldots,p,
\end{array}\right.
\end{equation}
with
\begin{equation}\label{eq:Pq_decomposition_gh}
g(x):= \frac{\sigma}{2}\|x\|^2+\langle q, x\rangle \quad \text{and}\quad h(x):=\frac{1}{2}\langle \left(\sigma I-Q\right)x, x \rangle.
\end{equation}
Observe that both functions $g$ and $h$ are strongly convex with parameters $\sigma$ and $\sigma-\lambda_{\max}(Q)$, respectively. Thus, \Cref{as:A1} holds for
\begin{equation}\label{eq:rho}
\rho:=\min\{\sigma, \sigma-\lambda_{\max}(Q)\},
\end{equation}
while \Cref{as:A2} trivially holds.

By using the indicator function $\iota_{\Fc}$ of the feasible set $\Fc$, we can rewrite problem~\eqref{eq:Pq} as an unconstrained nonsmooth DC optimization problem, which can be tackled by the DCA. In this case, the DCA becomes the projected gradient method from convex programming~\cite{TaoOMS}
\begin{equation}\label{DCA}
x_{k+1} = P_{\Fc} \left( x_k - \frac{1}{\sigma} (Qx_k+q)\right),
\end{equation}
where $P_{\Fc}$ denotes the projection mapping onto the feasible set $\Fc$.

To derive the R-linear convergence of the iterative sequence generated by the BDCA, we will use the following lemmas. The first one is a classical result regarding the connected components of the KKT set $\overline{\Fc}$ and can be found in \cite[Lemma 3.1]{LuoTseng1992}.

\begin{lemma}\label{TsengLemma}
	Let $\Fc_1, \Fc_2,\ldots,\Fc_r$ be the connected components of the KKT set $\overline{\Fc}$.  Then we have
	$$
	\overline{\Fc} = \bigcup_{i=1}^r \Fc_i,
	$$
	and the following properties hold:
	\begin{enumerate}
		\item each $\Fc_i$ is the union of finitely many polyhedral convex sets;\label{TsengLemmaI}
		\item the sets $\Fc_i$, $i=1,2,\ldots,r$, are properly separated from each others, that is \label{TsengLemmaII}
		$$
		\inf\left\{ d(x, \Fc_i), x \in \Fc_j \right\} >0,  \quad\text{for all } i \neq j;
		$$
		\item $\phi$ is constant on each $\Fc_i$.\label{TsengLemmaIII}
	\end{enumerate}
\end{lemma}

The second one is a local error bound result originally stated in \cite[Theorem 2.3]{LuoTseng1992} and  later extended in \cite[Lemma 2.1]{TuanJMAA}.
\begin{lemma}
	There exist scalars $\epsilon > 0$ and $\tau>0$ such that
	\begin{equation}\label{errorbound}
	d(x, \overline{\Fc}) \leq \tau \left\| x- P_{\Fc} \left( x - \frac{1}{\sigma} (Qx+q)\right) \right\|,
	\end{equation}
	for all $x \in \Fc$ with
	\begin{equation}\label{localCondition}
	\left \| x- P_{\Fc} \left( x - \frac{1}{\sigma} (Qx+q)\right) \right \| \leq \epsilon.
	\end{equation}
\end{lemma}

We are now in a position to establish the R-linear (aka geometric) convergence of the iterative sequence~$\{x_{k}\}$ generated by BDCA. Geometric convergence is a special type of R-linear convergence, which is implied by Q-linear convergence (see, e.g.~\cite[Section~5.2.1]{book}). The next result extends~\cite[Theorem~2.1]{TuanJMAA} to the BDCA. Its proof employs similar techniques, which are originally based on~\cite{LuoTseng1992}.

\begin{theorem}\label{linearRate}
	If \eqref{eq:Pq} has a solution and \Cref{as:A3} holds, then the sequence $\{x_{k}\}$ generated by BDCA converges geometrically to  a KKT point $\overline{x}$ of~\eqref{eq:Pq}, that is, there exist some constants $C>0$ and $\eta\in{]0,1[}$ such that
\begin{equation}\label{eq:Rlinear}
\|x_k-\overline{x}\|\leq C\eta^k,\quad\text{for all large }k.
\end{equation}%
\end{theorem}
\begin{proof}
First, observe that by~\eqref{eq:phi_decreasing} we have
$$(\rho+\alpha\lambda_k^2)\|d_k\|^2\leq\phi(x_k)-\phi(x_{k+1}).$$
By \Cref{prop:innerloop}\Cref{prop:innerloopI}, we know that the right-hand side of this inequality converges to zero as $k\to\infty$. Thus,
$$
\lim_{k \to \infty} \|y_{k}-x_k\| = \lim_{k \to \infty} \|d_k\| =   0 = \lim_{k \to \infty} \lambda_k \|d_{k}\|,
$$
which implies
\begin{equation}\label{eq:0}
\lim_{k \to \infty} \|x_{k+1}-x_k\| = \lim_{k \to \infty} (1+\lambda_k)\|d_k\| =   0.
\end{equation}
Let $\epsilon > 0$ and $\tau>0$ be such that \eqref{errorbound} and \eqref{localCondition} hold.
Since $y_k = P_{\Fc} \left( x_k - \frac{1}{\sigma} (Qx_k+q) \right)$, there exists some $k_0\in\N$ such that
$$
\left \Vert x_k-P_{\Fc} \left( x_k - \frac{1}{\sigma} (Qx_k+q) \right) \right \Vert= \| y_k -x_k\| \leq \epsilon,\quad \forall k \geq k_0.
$$
Hence, we obtain
\begin{equation*}
d(x_k, \overline{\Fc}) \leq \tau \left \Vert x_k- P_{\Fc} \left( x_k - \frac{1}{\sigma} (Qx_k+q)\right) \right \Vert =
\tau \|x_k -y_k \|  = \tau \|d_k\|, \quad \forall k \geq k_0.
\end{equation*}
Due to the fact that  $\overline{\Fc}$ is nonempty and closed, there exists $z_k \in \overline{\Fc}$, for each $k\in\N$, such that  ${d}(x_k, \overline{\Fc})  =  \|x_k - z_k\|$. Thus
\begin{equation}\label{V-2}
\|x_k - z_k\| \leq \tau \|d_k\|, \quad \forall k \geq k_0,
\end{equation}
which implies
\begin{equation}\label{eq:1}
\lim_{k \to \infty} \|x_k-z_k\|  = 0.
\end{equation}
Since
$$
\|z_{k+1}-z_k\| \leq \|z_{k+1}-x_{k+1}\| + \|x_{k+1}-x_{k}\|+\|x_{k}-z_{k}\|,
$$
it follows from~\eqref{eq:0} and~\eqref{eq:1} that
\begin{equation}\label{zkAssymtotic}
\lim_{k \to \infty} \|z_{k+1}-z_k\|  = 0.
\end{equation}
Now, let $\Fc_1, \Fc_2,\ldots,\Fc_r$ be the connected components of $\overline{\Fc}$. By \Cref{TsengLemma}\Cref{TsengLemmaII} and \eqref{zkAssymtotic} there exists  $\Fc_0 \in \{\Fc_1, \Fc_2,\ldots,\Fc_r \}$ and $k_1 \geq k_0$ such that $z_k \in \Fc_0$ for all $k \geq k_1$. The last assertion of Lemma \ref{TsengLemma} implies that
\begin{equation}
	\phi (z_k) = c,  \quad \forall k \geq k_1. \label{V-1}
\end{equation}
Note that, since $z_k$ is a KKT point of \eqref{eq:Pq}, we have $\left\langle  Qz_k +q, x_k - z_k\right\rangle \geq 0$. Then,
\begin{align*}
\phi(z_k)-\phi(x_k)&=\frac{1}{2}\langle Qz_k,z_k\rangle -\frac{1}{2}\langle Qx_k,x_k\rangle + \langle q,z_k-x_k\rangle\\
&\leq \frac{1}{2}\langle Qz_k,z_k\rangle -\frac{1}{2}\langle Qx_k,x_k\rangle + \langle Qz_k,x_k-z_k\rangle\\
&= \frac{1}{2}\langle Q(x_k-z_k),z_k-x_k\rangle \\
&\leq \frac{1}{2}\|Q\|\|x_k-z_k\|^2.
\end{align*}
From \Cref{prop:innerloop}\Cref{prop:innerloopI} we know that
$\lim_{k \to \infty}  \phi(x_k) = \bar\phi$. Hence, for all $k \geq k_1$,
\begin{equation}\label{V0}
	c = \phi(z_k)   \leq \phi(x_k) + \frac{1}{2}\| Q \|  \| z_k - x_k\| ^2 \to \bar\phi, \text{ as } k\to\infty.
\end{equation}%

We prove now that $\bar\phi \leq c$. Indeed, on the one hand,
from \eqref{eq:phi_decreasing} and \eqref{V-1},
 we have for all $k \geq k_1$ that
\begin{align}
		\phi(x_{k+1}) - c &\leq \phi(y_{k}) - c = \phi(y_{k}) - \phi (z_k) \nonumber\\
	&= \frac{1}{2} \left\langle Q y_{k}, y_{k}\right\rangle + \left\langle q, y_{k}\right\rangle
	-\frac{1}{2} \left\langle Q z_{k}, z_{k}\right\rangle - \left\langle q, z_{k}\right\rangle \nonumber\\
	&= \left\langle Q z_{k} + q, y_{k} - z_{k}\right\rangle +\frac{1}{2} \left\langle Q (y_{k}-z_k), y_{k}-z_k\right\rangle \nonumber\\
	&\leq \left\langle Q z_{k} + q, y_{k} - z_{k}\right\rangle +\frac{1}{2} \| Q\| \|y_{k}-z_k\|^2. \label{V1}
\end{align}
On the other hand,
since
$y_k = P_{\Fc} \left( x_k - \frac{1}{\sigma} (Qx_k+q) \right)$ and $z_k \in \Fc$, we deduce (see, e.g., \cite[Theorem~3.16]{BC17}) that
$$
\left\langle  x_k - \frac{1}{\sigma} (Qx_k+q) - y_k, y_k - z_k \right\rangle \geq 0.
$$
Therefore,
\begin{align}
 \left\langle Q z_{k} + q, y_{k} - z_{k}\right\rangle   \nonumber
&=  \left\langle Q x_{k} + q, y_{k} - z_{k}\right\rangle  + \langle  Q(z_k- x_k), y_k- z_k\rangle \\\nonumber
&\leq \sigma\langle x_k - y_k , y_k - z_k  \rangle + \|Q\|\|x_k -z_k\|\|y_k - z_k\| \\
 &\leq \left(\sigma \|x_k - y_k\| +\|Q\| \|x_k -z_k\|\right)  \|y_k - z_k\|.
 \label{V2}
\end{align}
Combining \eqref{V1} and \eqref{V2}, we obtain for all $k\geq k_1$ that
\begin{equation}
	\phi(x_{k+1}) - c  \leq \left(\sigma  \|x_k - y_k\|   + \|Q\| \|z_k -x_k\|  + \frac{1}{2} \| Q\| \|y_{k}-z_k\|\right)\|y_k - z_k \|. \label{V3}
\end{equation}
From \eqref{V-2}, we have
$$
\|y_k-z_k\|   \leq \|y_k-x_k\| + \|x_k-z_k\| \leq (1+\tau) \|d_k\|,  \quad  \forall k \geq k_1.
$$
Hence, we deduce from~\eqref{V3} that
\begin{equation}
\phi(x_{k+1}) - c  \leq \beta \|d_k\|^2, \quad  \forall k \geq k_1, \label{V4}
\end{equation}
where $\beta :=  (1+\tau) \left( \sigma+ \|Q\| \tau  + \frac{1+\tau}{2} \|Q\| \right) $.
Passing \eqref{V4} to the limit as $k \to \infty$, we get
$$
\bar\phi = \lim_{k \to \infty} \phi(x_{k+1}) \leq c.
$$
The latter together with \eqref{V0} imply that $\bar\phi = c$. Therefore, it follows from \eqref{V4} and \eqref{eq:phi_decreasing} that
\begin{equation*}
\phi(x_{k+1}) - \bar\phi  \leq \beta \|d_k\|^2 \leq  \frac{\beta}{\rho}  \left( \phi(x_k )- \phi(x_{k+1}) \right),   \quad  \forall k \geq k_1, \label{V5}
\end{equation*}
or, equivalently,
\begin{equation*}
\left( 1+   \frac{\beta}{\rho} \right)  \left( \phi(x_{k+1}) - \bar\phi \right)
 \leq   \frac{\beta}{\rho}  \left( \phi(x_k )- \bar\phi \right),   \quad  \forall k \geq k_1. \label{V6}
\end{equation*}
Since $\left\lbrace \phi(x_k ) \right\rbrace $ is monotonically decreasing to $\bar\phi$, we deduce from the last inequality that
\begin{align}
\phi(x_{k}) - \phi(x_{k+1})  & = (\phi(x_{k})-\bar\phi)  +(\bar\phi-\phi(x_{k+1})) \nonumber\\
&\leq \phi(x_k )- \bar\phi  \leq \left( \phi(x_{k_1})- \bar\phi \right) \left(\frac{\beta}{\rho+\beta}\right)^{k-k_1} \nonumber\\
 &= M_0 \eta^{2k}, \qquad \forall k\geq k_1,\label{V7}
\end{align}
where $M_0:= \left( \phi(x_{k_1})- \bar\phi \right) (\rho+\beta)^{k_1}\beta^{-k_1} $ and $\eta  := \sqrt{\frac{\beta}{\rho + \beta}} \in {]0,1[}$.
Consider now
$$M_1:=\max_{\lambda\geq 0}\frac{(1+\lambda)^2}{\alpha\lambda^2+\rho}=\frac{\alpha+\rho}{\alpha\rho}.$$
Hence, from \eqref{eq:phi_decreasing} and \eqref{V7}, we obtain
\begin{align*}
\|x_{k+1} - x_k\|^2 & = (1+ \lambda_k)^2 \|d_k\|^2\\
&\leq  \frac{(1+ \lambda_k)^2}{\alpha\lambda_k^2+\rho} \left( \phi(x_{k}) - \phi({x_{k+1}})\right)\\
&\leq  M_1M_0  \eta^{2k},  \quad  \forall k \geq k_1,
\end{align*}
which implies
$$
\|x_{k+1} - x_k\| \leq M \eta^{k},  \quad  \forall k \geq k_1,
$$
with $M := \sqrt{M_1M_0}$. 
Then, for all $k\geq k_1$ and $m \geq 1$, we have
\begin{align}
\|x_{k+m} - x_k\|
&\leq  \|x_{k+m} - x_{k+m-1}\|+ \cdots +\|x_{k+1} - x_{k}\|\nonumber \\
&\leq (\eta^{m-1} + \cdots + \eta+1) M \eta^k \leq \frac{M}{1-\eta} \eta^k.\label{V8}
\end{align}
This implies that $\left\lbrace x_k \right\rbrace $ is a Cauchy sequence and hence converges to a point $\overline{x} \in  \Fc$. According to \Cref{prop:innerloop}\Cref{prop:innerloopII} and \Cref{rem:int}, combined with the fact that $\dom h=\RR^n$ is an open set, we must have that $\overline{x}$ is a KKT point of~\eqref{eq:Pq}. Moreover, passing the inequality \eqref{V8} to the limit as $m \to \infty$, we obtain
$$
\|\overline{x} - x_k\| \leq \frac{M}{1-\eta} \eta^k ,\quad\forall k\geq k_1,
$$
which concludes the proof.
\end{proof}

\section{Numerical experiments}\label{sec:application1}
The purpose of this section is to compare the performance of BDCA (\Cref{alg:BDCA}) against the classical DCA. We refrain from comparing the algorithms against other methods, as our only aim here is to show that it is more advantageous to apply BDCA than DCA whenever this is possible, not to show that these methods are the most efficient for the problems that we consider. Consequently, we tested both algorithms in the three different settings that can occur, depending on whether the feasible set is unbounded, bounded with general constraints, and bounded with box constraints (so the feasibility in the line search step can be ensured). We applied both algorithms for testing copositivity~\cite{Durr} (unbounded) and for solving the $\ell_1$ and $\ell_\infty$ trust-region subproblem~\cite{Conn, Gratton2008} (bounded, with box constraints in the case of $\ell_\infty$).   In our last experiment we test the performance on piecewise quadratic problems with box constraints, whose objective function is thus nonsmooth. All these problems clearly satisfy Assumptions~\ref{as:A1}-\ref{as:A3} and admit a DC decomposition with $g$ being a quadratic function. Thus, subproblem $(\mathcal{P}_k)$ can be solved by computing a projection onto the feasible set $\mathcal{F}$, analogous to that stated in equation \eqref{DCA}. The projector onto the feasible set in our applications can be directly computed.

All the codes were written in Python 3.7 and the tests were run on an Intel Core i7-4770 CPU 3.40GHz with 32GB RAM, under Windows 10 (64-bit).

\subsection{Testing copositivity}
Recall that a given $n \times n $ matrix $A$ is said to be \emph{copositive} if
$$\langle Ax, x\rangle\geq 0,\quad \text{for all } x\in\RR^n_{+},$$
where $\mathbb{R}^n_{+}$ stands for the non-negative orthant. Copositivity has recently attracted considerable attention in mathematical optimization, see e.g.~\cite{Bomze,Burer,Durr,Nie}.
The problem of determining whether a given matrix is not copositive is known to be NP-complete~\cite{Murty} and it can be recast as the following non-convex optimization problem
\begin{equation} \label {problemDefine}
\min_{x \in \mathbb{R}^n_{+}} \; \phi(x):= \langle Ax, x\rangle.
\end{equation}
The copositivity of $A$ is now  equivalent to $\min_{x \in \mathbb{R}^n_{+}} \, \phi(x) = 0$. In \cite{Durr}, the authors reformulated \eqref{problemDefine} as a DC problem according to the decomposition in \eqref{eq:Pq_decomposition}-\eqref{eq:Pq_decomposition_gh}, and applied DCA as a heuristic for testing whether a matrix is not copositive. To be more specific, given $\sigma> \max\left\lbrace \lambda_{\max}(A),0\right\rbrace $, the reformulation of problem \eqref{problemDefine} as a DC problem becomes
\begin{equation*} \label {problemDC}
\left\{\begin{array}{ll}
\displaystyle\min_{x\in\RR^n}& g(x) - h(x)=\phi(x) \vspace{1ex}\\
\text{s.t.} &x_i\geq 0,~i=1,\ldots,n,
\end{array}\right.
\end{equation*}
with
\begin{equation*}
g(x):= \frac{\sigma}{2}\|x\|^2 \quad \text{and}\quad h(x):=\frac{1}{2}\langle\left(\sigma I-A\right)x,x\rangle.
\end{equation*}
Under this decomposition, DCA is applied as a heuristic to determine the copositivity of a given matrix as follows: if at some iterate $\phi(x_k)<0$, then the matrix is non-copositive; otherwise, if a critical point is reached, the instance is undecidable.

Copositive matrices play an important role in graph theory. The size of the largest complete subgraph contained in a given graph $G$, denoted by $\gamma(G)$, is known as the \emph{clique number} of $G$. If $A$ and $E$ are the adjacency matrix of $G$ and the matrix of all ones, respectively, it can be shown (see~\cite[Corollary~2.4]{KP02}) that
\begin{equation*}
\gamma(G)=\min \left\{ \mu : \mu(E-A)-E \text{ is copositive}\right\}.
\end{equation*}
Therefore, the matrix $\mu(E-A)-E$ will be copositive if $\mu\geq \gamma(G)$ and non-copositive otherwise. Furthermore, in the latter case, the matrix will be closer to the copositive cone as $\mu$ approaches $\gamma(G)$ from the left.

In our tests we considered matrices constructed as follows. Let $G$ be the cycle graph of $n$ nodes whose adjacency matrix, $A_{\text{cycle}}=(a_{ij})\in\RR^{n\times n}$, is given component-wise by
\begin{equation*}\label{eq:Acycle}
a_{ij}:=\left\{\begin{array}{rl}
1, & \text{if } \lvert i-j\rvert\in\{1,n-1\},\\
0, & \text{otherwise.}\end{array}\right. 
\end{equation*}
Its clique number is clearly $\gamma(G)=2$. Hence, the matrix
\begin{equation}\label{eq:Qn}
Q_n^{\mu}:=\mu(E-A_{\text{cycle}})-E\in\RR^{n\times n}
\end{equation}
is copositive for all $\mu\geq 2$ and non-copositive for $\mu<2$. In fact, when $\mu=2$ it coincides with the so-called Horn matrix $H_n$ (see, e.g., \cite[Section~4]{JRcopos}). For instance, the Horn matrix $H_5$ takes the form
\begin{equation*}
H_5:=Q_5^2=\left(\begin{array}{rrrrr}
 1 & -1 &  1 &  1 & -1\\
-1 &  1 & -1 &  1 &  1\\
 1 & -1 &  1 & -1 &  1\\
 1 &  1 & -1 &  1 & -1\\
-1 &  1 &  1 & -1 &  1\\
\end{array}\right).
\end{equation*}

\paragraph{Experiments}

In our numerical tests we used the parameter setting as
$$\alpha:=0.01 \quad \text{and}\quad  \beta:=0.1.$$
The trial step-size $\overline{\lambda}_k$ in the boosting step of BDCA (Line 8 of \Cref{alg:BDCA}) was chosen to be self-adaptive as in~\cite{nBDCA}. This technique proceeds as follows. At the first iteration, choose any $\overline{\lambda}_0>0$. Then, for $k\geq 1$, if the line search has never been used, we take $\overline{\lambda}_k=\overline{\lambda}_0$. Otherwise, if the two previous trial step-sizes have been directly accepted (without being reduced by the  backtracking step), then the last accepted positive $\lambda$ is scaled by a factor of $\gamma>1$ and used as the current trial step-size. If that is not the case, the trial step-size is set as the last positive value of $\lambda$ accepted in previous iterations.  In our tests we used $$\overline{\lambda}_0:=1 \quad\text{and}\quad\gamma:=2.$$

In our first numerical experiment, we considered Horn matrices of different sizes, $H_n$, for $n\in\{1000,1250,\ldots,5000\}$. For each size, DCA and BDCA were run from the same $100$ starting points randomly generated in the intersection of the non-negative orthant with the unit ball. We stopped the algorithms when $\norm{d_k} \leq 10^{-9}$ for the first time.
The results are shown in~\Cref{fig:ExpHorn}, where we can observe that, on average, BDCA was more than 15 times faster than DCA for all sizes. As expected, since Horn matrices are copositive, both algorithms converged to critical points with a positive objective value very close to 0. It is worth to mention that the objective function at the points found by BDCA was usually smaller than at the ones found by DCA.
We also show in \Cref{fig:ExpQHorn_count} the percentage of iterations at which the boosting step {was} activated, that is, when condition in line 7 of \Cref{alg:BDCA} {was} fulfilled. We observed that, for all sizes, the linesearch {was} performed in around the 45\% of the iterations.
In \Cref{fig:HornInst} we show the behavior of both algorithms in a particular instance for testing the copositivity of $H_{1000}$.

\begin{figure}[ht!]
\centering
\includegraphics[width=0.7\textwidth]{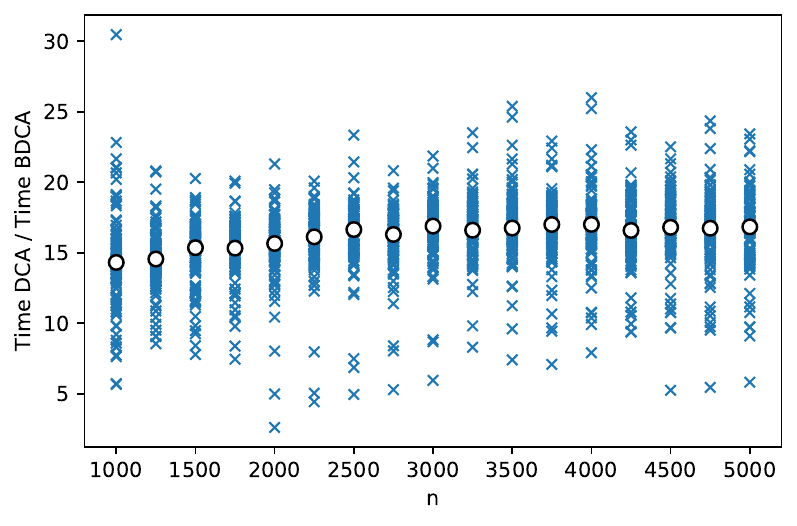}

\vspace{-2ex}
\caption{Comparison between DCA and BDCA for checking the copositivity of Horn matrices of order $n\in\{1000,1250,\ldots,5000\}$. For each size, we represent the ratios of the running time between DCA and BDCA for $100$ random starting points (blue crosses) and the median ratio among all of them (white circle).}\label{fig:ExpHorn}
\end{figure}

\begin{figure}[ht!]
\centering
{\includegraphics[width=0.6\textwidth]{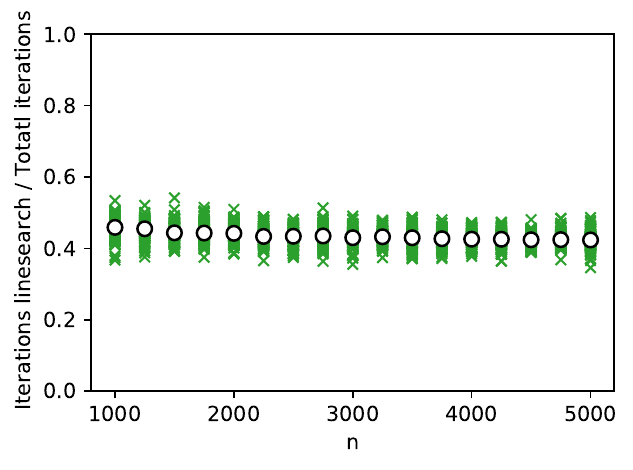}}
\vspace{-2ex}
\caption{Ratio of the number of iterations at which the boosting step of BDCA {was} activated with respect the number of iterations performed for testing the copositivity of Horn matrices of order $n\in\{1000,1250,\ldots,5000\}$. For each size, we represent this ratio for $100$ random starting points (green crosses) and the median ratio among all of them (white circle).}\label{fig:ExpQHorn_count}
\end{figure}

\begin{figure}[ht!]
\centering
\includegraphics[width=0.8\textwidth]{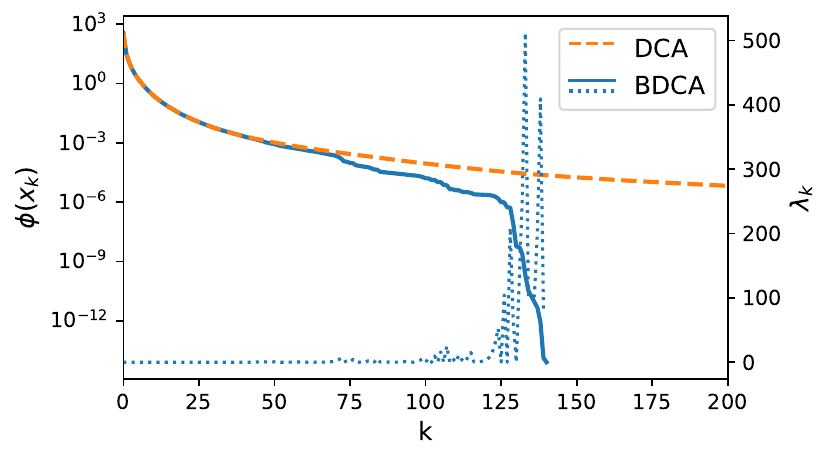}

\vspace{-2ex}
\caption{Value of the objective function of DCA and BDCA (using logarithmic scale in the left axis) as well as the step-size used in BDCA (right axis, dotted blue line), with respect to the iteration, for checking the copositivity of the Horn matrix of order $n=1000$ from the same random starting point.}\label{fig:HornInst}
\end{figure}

In our second experiment we considered matrices of the form $Q_n^\mu$ as defined in~\eqref{eq:Qn}. In order to generate hard instances (those which are close to be copositive)  we took $\mu:=1.9$. For each size $n\in\{1000, 1250,\ldots, 5000\}$, DCA and BDCA were run from the same $100$ random starting points generated as in our previous experiment. In this case, we let the algorithms run until they find a negative objective value (which exists because of the non-copositivity of the matrices). We used two stopping criteria, whose results are depicted in~\Cref{fig:ExpQHorn}: on the left, the algorithms were stopped when any negative objective value was found; on the right, the objective value was required to be smaller than $-10^{-4}$. We do not show any results on the second criterion for $n$ greater than $2000$ because DCA becomes extremely slow (for $n=2000$, the instances solved by DCA required more than 5 minutes on average). This time the advantage of BDCA with respect to DCA increased with the size~$n$, and was significantly greater than in the previous experiment when the second criterion was used.

\begin{figure}[ht!]
\centering
\subfigure[Stopping criterion: $\phi(x)<0$]{\includegraphics[width=0.495\textwidth]{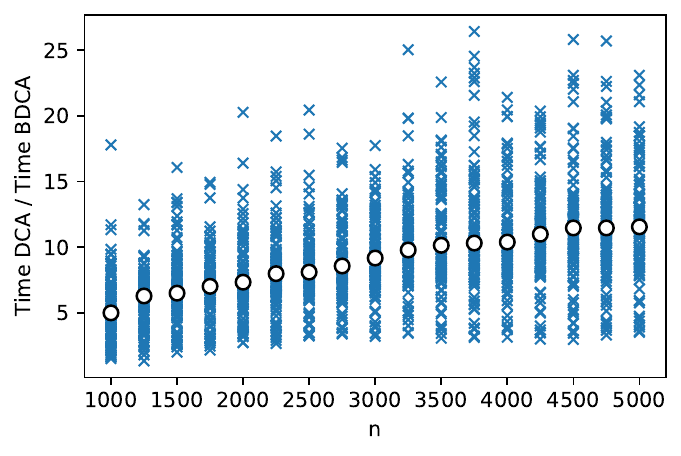}}
\subfigure[Stopping criterion: $\phi(x)<-10^{-4}$]{\includegraphics[width=0.495\textwidth]{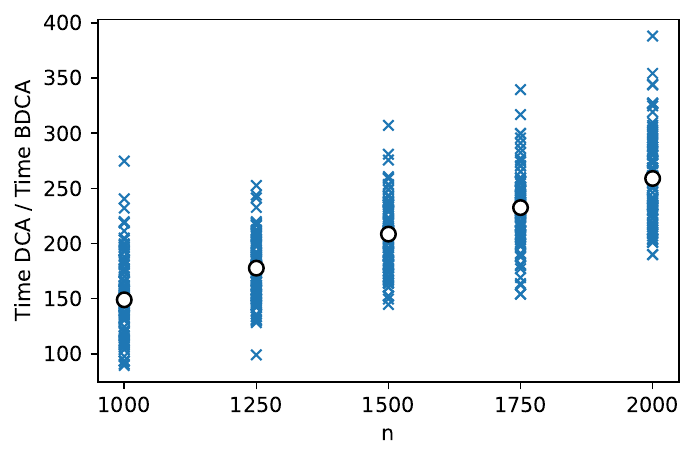}}
\vspace{-2ex}
\caption{Comparison between DCA and BDCA for detecting the non-copositivity of matrices of various orders $n$. For each size, we represent the ratios of the running time  between DCA and BDCA for $100$ random starting points (blue crosses) and the median ratio among all of them (white circle).}\label{fig:ExpQHorn}
\end{figure}

\begin{remark}
Clearly, BDCA iterations will be more time consuming than those of DCA due to the need of checking condition in Line 7 and the linesearch procedure. Note that the linesearch requires of function evaluations of $f$, which may be expensive at some applications. For this reason, in our experiments we compare the total CPU running time of the algorithm, as it includes the possibly wasted extra time of the boosting step.
\end{remark}

\subsection{Solving the $\ell_1$ and $\ell_\infty$ trust-region subproblem}\label{sec:trsp}
The trust-region subproblem {(TRSP)} arises in trust-region methods, which are optimization algorithms that consist in replacing the original objective function by a model which is a good approximation of the original function at the current iterate. The models are usually defined to be a quadratic function, in which case the task consists in solving nonconvex quadratic optimization problems of the type
$$\min_{\vertiii{x}\leq r}\frac{1}{2}\langle Ax, x\rangle+\langle b,x\rangle,$$
where $A$ is an $n\times n$ real symmetric matrix, $b\in\RR^n$, $r$ is a positive number and $\vertiii{\cdot}$ is any given norm. Of course, the choice of the norm has a serious impact on the difficulty for solving the subproblems.

The application of the DCA for solving the Euclidean trust-region subproblem (the most common choice for the norm) was first proposed in~\cite{TT98} and its convergence was further studied in~\cite{An2012, Tuan2012, Tuan2013}. Thanks to the structure of the trust-region subproblem, the implementation of the DCA is very simple and efficient, as the iterations are given by~\eqref{DCA} and only require matrix-vector products. On the other hand, as observed in~\cite[Section~7.8]{Conn} and~\cite{Gratton2008}, the choice of the Euclidean norm for the definition of the trust-region has several drawbacks, especially when the problem at hand is box-constrained, in which case the intersection of the Euclidean ball of the trust region with the feasible set has a complicated structure. The choice of the $\ell_\infty$ norm $\|\cdot\|_\infty$ for the trust-region when the problem involves bounds of the type $l\leq x\leq u$ permits to ensure feasibility by simply requiring similar bounds on the trust-region subproblem. Unfortunately, the  $\ell_1$ and $\ell_\infty$ trust-region subproblem are NP-hard~\cite{Murty}, a class of problems for which it is commonly believed that there are no  polynomial time algorithms.

In this subsection we compare the performance of DCA and BDCA on the challenging $\ell_1$ and $\ell_\infty$ trust-region subproblems. To this aim, we consider the DC decomposition~\eqref{eq:Pq_decomposition}-\eqref{eq:Pq_decomposition_gh} of the (possibly) nonconvex part of the objective function. Thus, given $\sigma> \max\left\lbrace \lambda_{\max}(A),0\right\rbrace$, the subproblems for the $\ell_1$ norm take the form
\begin{equation}\tag{$\mathcal{P}_1$}\label{eq:P1}
\left\{\begin{array}{ll}
\displaystyle\min_{x\in\RR^n}& g(x) - h(x)=\phi(x) \vspace{1ex}\\
\text{s.t.} &\sum_{i=1}^n\lvert x_i\rvert\leq r,
\end{array}\right.
\end{equation}
with
\begin{equation*}
g(x):= \frac{\sigma}{2}\|x\|^2+\langle b, x\rangle \quad \text{and}\quad h(x):=\frac{1}{2}\langle\left(\sigma I-A\right)x,x\rangle,
\end{equation*}
while for the $\ell_\infty$ norm these problems are
\begin{equation}\tag{$\mathcal{P}_\infty$}\label{eq:Pinf}
\left\{\begin{array}{ll}
\displaystyle\min_{x\in\RR^n}& g(x) - h(x)=\phi(x) \vspace{1ex}\\
\text{s.t.} &-r\leq x_i\leq r,~i=1,\ldots,n.
\end{array}\right.
\end{equation}
Observe that the feasible set of~\eqref{eq:P1} can be equivalently written in terms of $2^n$ linear constraints of the type $\pm x_1\pm x_2\pm\cdots\pm x_n\leq r$, so the problem is a particular instance of~\eqref{eq:P}. In principle, as mentioned in \Cref{rem:maxlambda}, the largest step-size derived in~\cite{FSS21} ensuring feasibility could be computed. However, as the feasible set has $2^n$ linear constraints, computing~\eqref{lambda_bound} would be very time-consuming.  Nonetheless, given a point $y_k$ and a feasible direction $d_k$, an upper bound of the maximum value of the step-size $\overline{\lambda}_k$ for~\eqref{eq:P1} is given by the Euclidean norm, since
$$\left\|y_k+\lambda d_k\right\|_1\geq\left\|y_k+\lambda d_k\right\|> r,\text{ for all }\lambda>\widehat{\lambda}_k,$$
where $$\widehat{\lambda}_k:=\frac{\langle y_k,d_k\rangle+\sqrt{\langle y_k,d_k\rangle^2-\|d_k\|^2
(\|y_k\|^2-r^2)}}{\|y_k\|^2},$$
so one must always choose $\overline{\lambda}_k\leq\widehat{\lambda}_k$ to avoid extra time in checking feasibility. On the other hand, for the $\ell_\infty$ norm, the situation is more favorable, as feasibility can be guaranteed whenever
$$0\leq\lambda\leq\min_{i\not\in I(x_k)}\left\{\frac{r}{\lvert d_{k,i}\rvert}-\frac{y_{k,i}}{d_{k,i}}\;\bigg\vert\; d_{k,i}\neq 0\right\},$$
which coincides with~\eqref{lambda_bound}, so no time will be wasted in Step~8 of Algorithm~\ref{alg:BDCA} if one takes $\overline{\lambda}_k$ satisfying the latter inequalities.

\paragraph{Experiments}

We used the same parameter setting as in the previous section for \cref{alg:BDCA}, except for the value of $\gamma$, which was increased to $20$. The matrix $A$ and the vector $b$ were generated with coordinates randomly and uniformly chosen in~${]-1,1[}$, while the value of $r$ was randomly and uniformly chosen in the range ${]0,\sqrt{n}/4[}$ for~\eqref{eq:P1} and ${]0,1/4[}$ for~\eqref{eq:Pinf}. For each size $n\in\{1000,1500,\ldots,5000\}$, DCA and BDCA were run from the same $100$ starting points randomly picked inside the trust-region. We stopped the algorithms when $\norm{d_k}/\|x_k\| \leq 10^{-8}$ for the first time. The time comparison for both norms are summarized in~\Cref{fig:TRSP}. BDCA was consistently faster than DCA. On average, it was 3.8 times faster for solving~\eqref{eq:P1} and 3.65 times faster for~\eqref{eq:Pinf}.

\begin{figure}[ht!]
\centering
\includegraphics[width=0.49\textwidth]{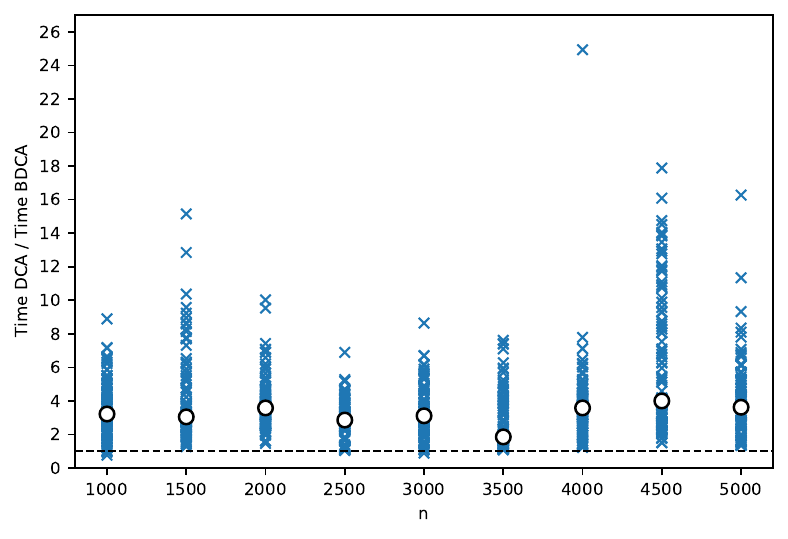}
\includegraphics[width=0.49\textwidth]{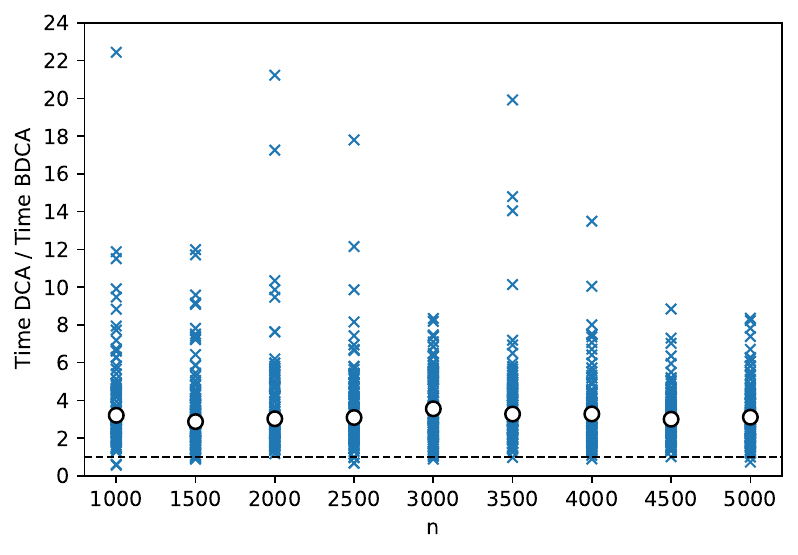}
\caption{Comparison between DCA and BDCA for solving randomly generated $\ell_1$ (left) and $\ell_\infty$ (right) trust-region subproblems in $\RR^n$, with $n\in\{1000,1250,\ldots,5000\}$. For each size, we represent the ratios of the running time between DCA and BDCA for $100$ random starting points (blue crosses) and the median ratio among all of them (white circle).}\label{fig:TRSP}
\end{figure}

{We also show in \Cref{fig:TRSP_count} the percentage of iterations at which the boosting step was activated. We observe that, for all sizes, the linesearch was performed on average in around 80\% of the iterations for the case of the $\ell_1$ norm, and 40\% for the $\ell_\infty$ norm.}

\begin{figure}[ht!]
\centering
\includegraphics[width=0.49\textwidth]{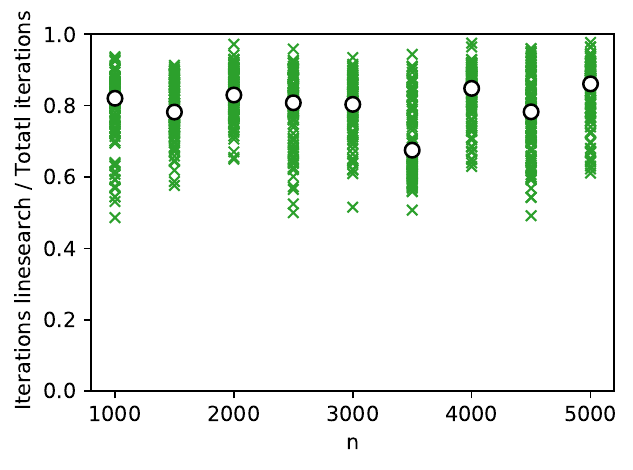}
\includegraphics[width=0.49\textwidth]{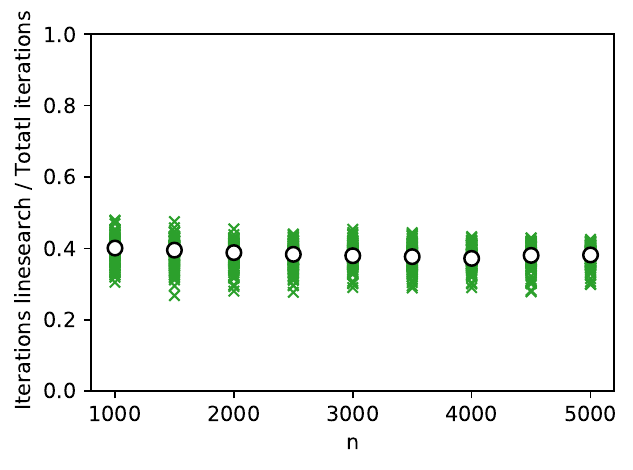}
\caption{Ratio of the number of iterations at which the boosting step of BDCA was activated with respect the number of iterations performed for solving randomly generated $\ell_1$ (left) and $\ell_\infty$ (right) trust-region subproblems in $\RR^n$, with $n\in\{1000,1250,\ldots,5000\}$. For each size, we represent this ratio for $100$ random starting points (green crosses) and the median ratio among all of them (white circle).}\label{fig:TRSP_count}
\end{figure}

When applied to~\eqref{eq:P1}, both algorithms obtained the same value in 875 instances, BDCA reached a smaller objective value in 18 instances, while DCA did so in 7. When applied to~\eqref{eq:Pinf}, BDCA obtained a smaller objective value in 461 instances, while DCA did so in 417 (the value was the same in 22 instances). In \Cref{fig:TRSPInst} we show the behavior of both algorithms for a particular instance with $n=1000$, which was chosen so that both algorithms attained the same objective value.

\begin{figure}[ht!]
\centering
\includegraphics[width=0.9\textwidth]{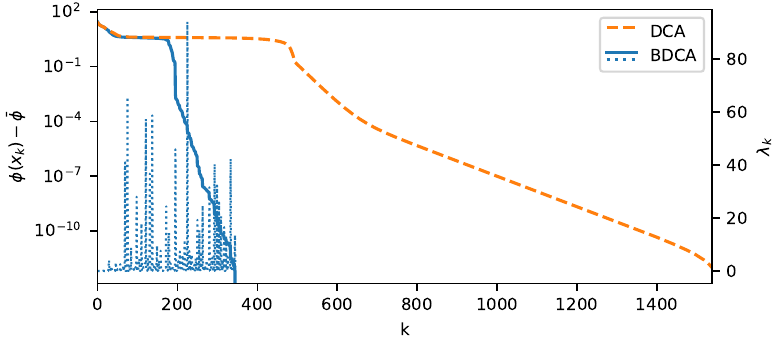}
\includegraphics[width=0.9\textwidth]{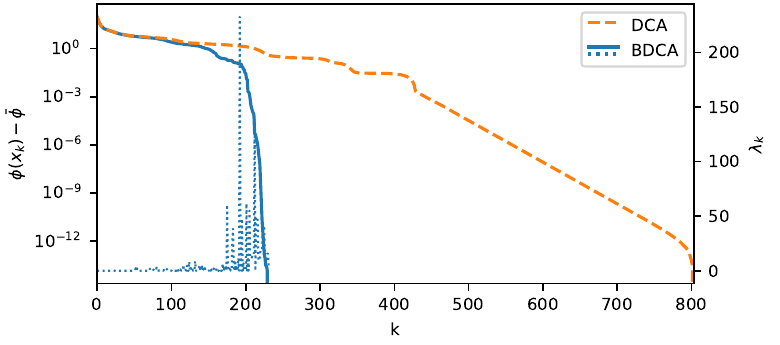}
\caption{Value of the objective function of DCA and BDCA (using logarithmic scale in the left axis) as well as the step-size used in BDCA (right axis, dotted blue line), with respect to the iteration, for solving a randomly generated $\ell_1$ (top) and $\ell_\infty$ (bottom) trust-region subproblem in $\RR^{1000}$ from the same random starting point. Both algorithms attained the same objective value $\bar\phi$.}\label{fig:TRSPInst}
\end{figure}

We conclude this section with some comments about the rate of convergence of DCA and BDCA. According to \Cref{linearRate}, both methods are R-linearly convergent, so by~\eqref{eq:Rlinear},
$$\lim_{k\to\infty}\|x_k-\overline{x}\|^{\frac{1}{k}}\leq \eta,$$
for some $\eta\in{]0,1[}$. In Figure~\ref{fig:TRSPInst_convergence},  for a particular random instance, we have represented the sequence $\{\|x_k-\overline{x}\|^{1/k}\}$, which was computed for both algorithms after obtaining the limit point $\overline{x}$ with a higher precision. We observe that the value of $\eta$ is very close to $1$ for DCA, while line-searches clearly helped improving this slow rate. A possible explanation about the bad R-linear convergence rate of DCA can be deduced from the proof of \Cref{linearRate}. The upper bound obtained there is $\frac{\beta}{\rho+\beta}$, with $\beta>\sigma$. When $\lambda_{\max}(Q)>0$, Assumption~\ref{as:A1} holds for $\rho$ as in~\eqref{eq:rho}, which would then be equal to $\sigma-\lambda_{\max}(Q)$. Therefore,
$$\frac{\beta}{\rho+\beta}=\frac{\beta}{\sigma-\lambda_{\max}(Q)+\beta}>\frac{\sigma}{2\sigma-\lambda_{\max}(Q)},$$
where the last inequality holds since the left term is an increasing function with respect to $\beta$.
In our numerical tests we took $\sigma-\lambda_{\max}(Q)=0.01$ (we numerically observed how larger values slowed down both DCA and BDCA). In the random instances shown in \Cref{fig:TRSPInst_convergence} $\lambda_{\max}(Q)$ was equal to $25.77$ for the $\ell_1$ norm and $25.89$ for the $\ell_\infty$ norm, which would give upper bounds on the R-linear convergence rate of $0.9996$ for both problems. According to the figures, these bounds seem to be tight, particularly for the $\ell_\infty$ norm.

\begin{figure}[ht!]
\centering
\includegraphics[width=0.8\textwidth]{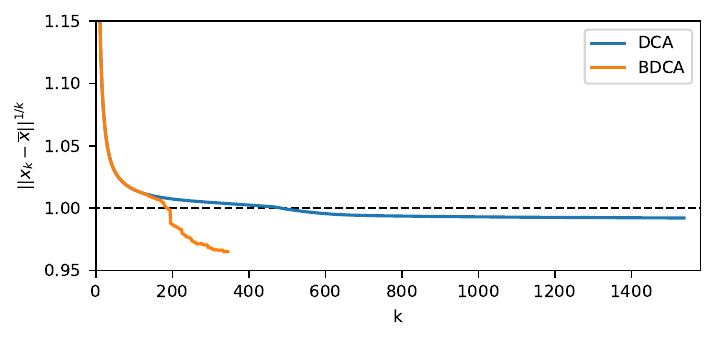}
\includegraphics[width=0.8\textwidth]{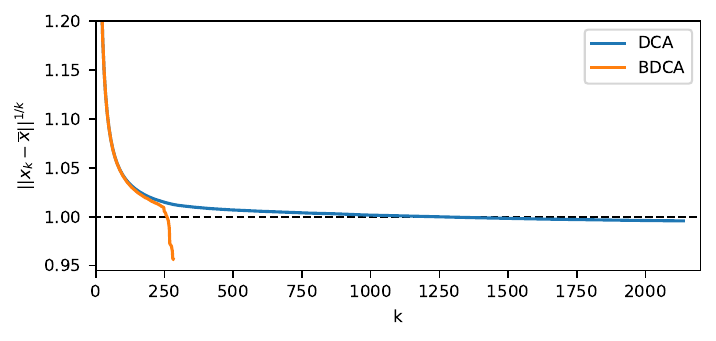}
\caption{Comparison on the rate of R-linear convergence of DCA and BDCA for solving the same randomly generated $\ell_1$ (top) and $\ell_\infty$ (bottom) trust-region subproblems in $\RR^{1000}$ from Figure~\ref{fig:TRSPInst}.}\label{fig:TRSPInst_convergence}
\end{figure}

\subsection{Piecewise quadratic problems with box constraints}

Finally, we test BDCA on optimization problems with nonsmooth objective function. To this aim, we consider piecewise quadratic problems with linear constraints of the form
\begin{equation}\label{pquadprob}
\left\{\begin{array}{ll}
\displaystyle\min_{x\in\RR^n}& \phi(x):=\min_{j\in\Set{1,\ldots,m}}\Set{ \frac{1}{2}\|x-c^j\|^2} \vspace{1ex}\\
\text{s.t.} &l_i\leq x_i\leq u_i,~i=1,\ldots,n;
\end{array}\right.
\end{equation}
for given $l=(l_1,\ldots,l_n), u=(u_1,\ldots,u_n)\in\RR^n$ and $c^1,\ldots,c^m\in\RR^n$.

As discussed in \cite[\S~6.3]{Oliv10}, the objective function of {this} problem admits the DC decomposition $\phi(x)=g(x)-h(x)$ with
\begin{equation*}
g(x):= \frac{1}{2}\sum_{j=1}^m \|x-c^j\|^2 \quad \text{and}\quad h(x):=\max_{l\in\Set{1,\ldots,m}}\Set{\frac{1}{2}\sum_{j=1\atop j\neq l}^m\|x-c^j\|^2}.
\end{equation*}

\paragraph{Experiments}

For our numerical tests, we generated problems of the form \eqref{pquadprob} as follows. First, the vector $l\in\RR^n$ was generated with coordinates randomly chosen in~${]-5,5[}${. Then, we randomly generated a vector $e\in\RR^n$ with coordinates in ${]0,5[}$ and we set $u:=l+e$. Finally,} the points $c_1,\ldots,c_m\in\RR^n$ were generated so that they became unfeasible for all constraints {as follows: each component $i\in\{1,\ldots,n\}$ of each of these vectors was randomly picked in one of the intervals ${]l_i-10,l_i[}$ or ${]u_i,u_i+10[}$}. In our experiment, for each $n\in\{100,150,\ldots,1000\}$ and each $m\in\{100,150,\ldots,1000\}$, DCA and BDCA were run from the same $100$ starting points randomly picked inside the feasible set. We used the same parameter setting and stopping criterion for the algorithms as in the previous TRSP experiments. The time comparison is shown in~\Cref{fig:Pwq_grid}, where we represent the median ration between DCA and BDCA among the 100 instances. Detailed results for all 100 runs are shown in \Cref{fig:PwqN} for fixed $n=500$ and \Cref{fig:PwqM} for fixed $m=500$. From the results of this comparison we infer that the superiority of BDCA increases as $m$ does, while it stabilizes as $n$ increases. {As in previous experiments, we show in \Cref{fig:Pwq_count} the percentage of iterations at which the boosting step was activated.} The objective value attained by DCA and BDCA were the same in all instances.

\begin{figure}[ht!]
\centering
{\includegraphics[width=0.68\textwidth]{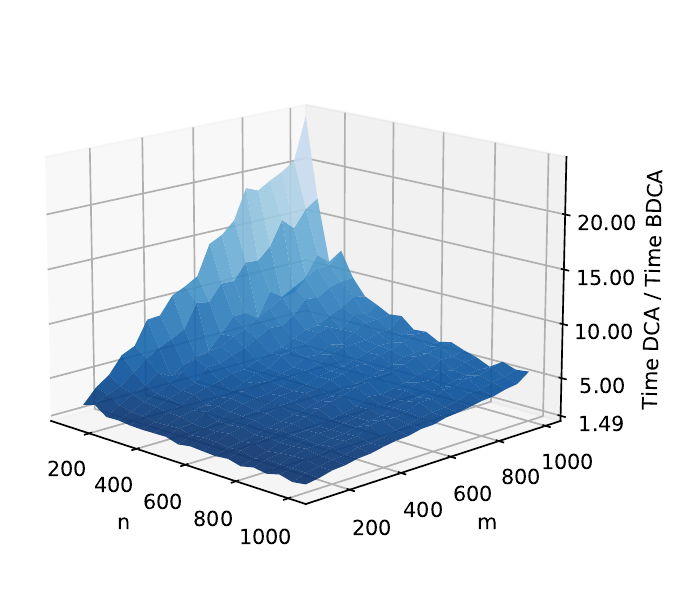}}
\vspace{-2ex}
\caption{Comparison between DCA and BDCA for solving randomly generated piecewise quadratic problems with $m$ pieces in $\RR^n$, for $n\in\{100,200,\ldots,1000\}$ and $m\in\{100,200,\ldots,1000\}$. For each pair $(n,m)$, we represent the median of the ratios of the running time between DCA and BDCA for $100$ random starting points.}\label{fig:Pwq_grid}
\end{figure}

\begin{figure}[ht!]
\centering
\subfigure[$n\in\Set{100,200,\ldots,1000}$ and $m=500$ \label{fig:PwqN}]{\includegraphics[width=0.47\textwidth]{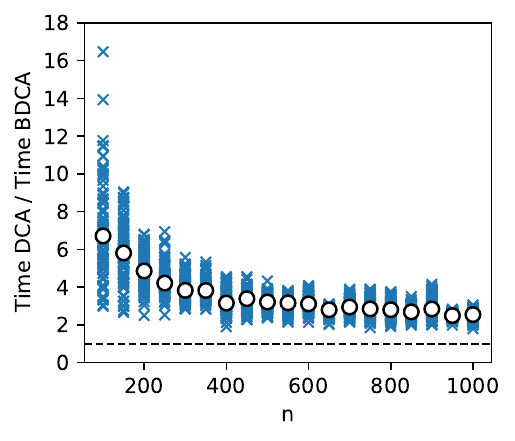}}%
\subfigure[$n=500$ and $m\in\Set{100,200,\ldots,1000}$ \label{fig:PwqM}]{\includegraphics[width=0.47\textwidth]{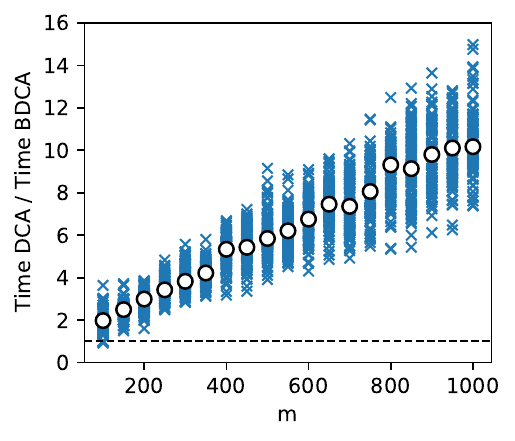}}
\caption{Comparison between DCA and BDCA for solving randomly generated piecewise quadratic problems with $m$ pieces in $\RR^n$, for $n\in\{100,200,\ldots,1000\}$ and $m=500$ (left), and for $n=500$ and $m\in\{100,200,\ldots,1000\}$ (right). For each problem, we represent the ratios of the running time between DCA and BDCA for $100$ random starting points (blue crosses) and the median ratio among all of them (white circle).}\label{fig:Pwq}
\end{figure}

\begin{figure}[ht!]
\centering
\subfigure[$n\in\Set{100,200,\ldots,1000}$ and $m=500$ \label{fig:PwqN_count}]{\includegraphics[width=0.47\textwidth]{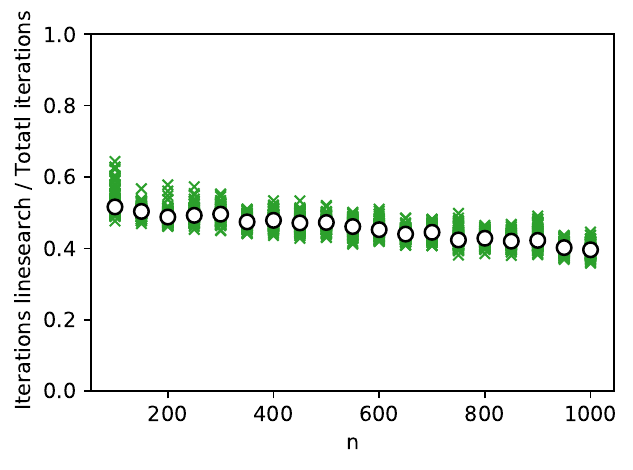}}%
\subfigure[$n=500$ and $m\in\Set{100,200,\ldots,1000}$ \label{fig:PwqM_count}]{\includegraphics[width=0.47\textwidth]{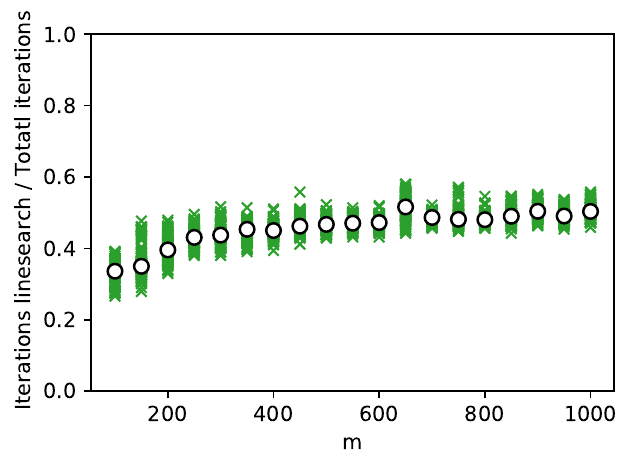}}
\caption{Ratio of the number of iterations at which the boosting step of BDCA was activated with respect the number of iterations performed for solving randomly generated piecewise quadratic problems with $m$ pieces in $\RR^n$, for $n\in\{100,200,\ldots,1000\}$ and $m=500$ (left), and for $n=500$ and $m\in\{100,200,\ldots,1000\}$ (right). For each size, we represent this ratio for $100$ random starting points (green crosses) and the median ratio among all of them (white circle).}\label{fig:Pwq_count}
\end{figure}

\section{Concluding remarks}\label{sec:conclusions}
We have extended the Boosted DC Algorithm for solving linearly constrained DC programming. The algorithm is proved to provide KKT points of the constrained problem. In addition, we have shown why this approach cannot be extended to more general convex constraints. The theoretical results were confirmed by some numerical experiments for testing the copositivity of matrices and for solving $\ell_1$ and $\ell_\infty$ trust-region subproblems. For copositive matrices, BDCA was on average fifteen times faster than DCA; for non-copositive ones, this advantage was much more superior; and for trust-region subproblems, BDCA was more than three times faster than DCA. We also considered piecewise quadratic problems with box constraints, which thus have nonsmooth objective functions. We observed again that BDCA was faster than DCA and, further, that the advantage {was more noticeable as the number of pieces of the objective function increased}. Future research includes investigation of alternative approaches to derive a Boosted DCA that permits to address any type of constrained DC programs. It would also be interesting to combine BDCA with the inertial technique employed in~\cite{OT19}.

\paragraph{Acknowledgments}
FJAA and RC were partially supported by the Ministry of Science, Innovation and Universities of Spain and the European Regional Development Fund (ERDF) of the European Commission (PGC2018-097960-B-C22), and by the Generalitat Valenciana (AICO/2021/165).
PTV was supported by Vietnam Ministry of Education and Training Project hosting by the University of Technology and Education, Ho Chi Minh City Vietnam (2023-2024). %

\bibliographystyle{plain}

\begin{thebibliography}{}


\bibitem{BDCA2018}
Arag{\'o}n~Artacho, F.J., Fleming, R., Vuong, P.T.:
\newblock Accelerating the {DC} algorithm for smooth functions.
\newblock {\em Math. Program.} 169(1), 95--118 (2018)

\bibitem{book}
Arag{\'o}n, F.J., Goberna, M.A., L{\' o}pez, M.A., Rodr{\'i}guez, M.M.L.:
\newblock {\em Nonlinear Optimization}.
\newblock Springer Undergraduate Texts in Mathematics and Technology (2019)

\bibitem{nBDCA}
Arag{\'o}n~Artacho, F.J., Vuong, P.T.:
\newblock The boosted difference of convex functions algorithm for nonsmooth functions.
\newblock {\em SIAM J. Optim.} 30(1), 980--1006 (2020)

\bibitem{BC17}
Bauschke, H.H., Combettes, P.L.:
\newblock {\em Convex analysis and monotone operator theory in Hilbert spaces, 2nd ed.}
\newblock Springer, Berlin (2017)

\bibitem{Bomze}
Bomze, I.M.:
\newblock Copositive optimization-{recent} developments and applications.
\newblock {\em European J. Oper. Res.} 216(3), 509--520 (2012)

\bibitem{Burer}
Burer, S.:
\newblock On the copositive representation of binary and continuous nonconvex quadratic programs.
\newblock {\em Math. Program.} 120(2), 479--495 (2009)


\bibitem{Conn}
Conn, A.R., Gould, N.I.M., Toint, P.L.:
\newblock {\em Trust-region methods}.
\newblock MPS/SIAM Series on Optimization (2000)

\bibitem{Durr}
D\"ur, M., Hiriart-Urruty, J.-B.:
\newblock Testing copositivity with the help of difference-of-convex optimization.
\newblock {\em Math. Program.} 140(1), 31--43 (2013)

\bibitem{FSS21}
Ferreira, {O.P.}, Santos, {E.M.}, Souza, {J.C.O.}:
\newblock Boosted scaled subgradient method for DC programming.
\newblock ArXiv: \href{2103.10757}{https://arxiv.org/abs/2103.10757} (2021)

\bibitem{FM81}
Fukushima, M., Mine, H.:
\newblock A generalized proximal point algorithm for certain non-convex minimization problems.
\newblock {\em Int. J. Syst. Sci.} 12(8), 989--1000 (1981)

\bibitem{Geremew2018} Geremew, W., Nam, N.M., Semenov, A., Boginski, V., Pasiliao, E.:
A DC programming approach for solving multicast
network design problems via the Nesterov smoothing
technique.
\newblock {\em J. Glob. Optim.} 72(4), 705--729 (2018)

\bibitem{Gratton2008} Geremew, S., Mouffe, M., Toint, P.L., Weber-Mendon\c{c}a, M.:
A recursive $\ell_\infty$-trust-region method for bound-constrained nonlinear optimization.
\newblock {\em IMA J. Numer. Anal.} 28(4), 827--861 (2008)


\bibitem{JRcopos}
Johnson, C.R., Reams, R.:
\newblock Constructing copositive matrices from interior matrices.
\newblock {\em Electron. J. Linear Al.}, 17, 9--20 (2008)

\bibitem{KP02}
de Klerk, E., Pasechnik, D.V.:
\newblock Approximation of the stability number of a graph via copositive programming.
\newblock {\em SIAM J. Optim.} 12(4), 875--892 (2002)

\bibitem{An2018}
Le~Thi, H.A., Pham Dinh, T.:
\newblock DC Programming and DCA: Thirty Years of Developments.
\newblock {\em Math. Program.} 169(1), 5--68 (2018)

\bibitem{An2012}
Le~Thi, H.A., Pham Dinh, T., Yen, N.D.:
\newblock Behavior of DCA sequences for solving the trust-region subproblem.
\newblock {\em J. Global Optim.} 53(2), 317--329 (2012)

\bibitem{tao2005dc}
Le~Thi, H.A., Pham~Dinh, T.:
\newblock The DC (difference of convex functions) programming and DCA revisited with DC models of real world nonconvex optimization problems.
\newblock {\em Ann. Oper. Res.} 133(1-4), 23--46 (2005)

\bibitem{JOTA2018}
Le~Thi, H.A., {Huynh V.N.}, Pham~Dinh, T.:
\newblock Convergence analysis of Difference-of-Convex Algorithm with subanalytic data.
\newblock{\em J. Optim. Theory Appl.} 179(1), 103--126 (2018)


\bibitem{LuoTseng1992}
Luo, Z.Q., Tseng, P:
\newblock Error bound and convergence analysis of matrix splitting algorithms for the affine variational inequality problem.
\newblock{\em  SIAM J. Optim. } 2(1), 43--54 (1992)

\bibitem{MF81}
Mine, H., Fukushima, M.:
\newblock A minimization method for the sum of a convex function and a continuously differentiable function.
\newblock {\em J. Optim. Theory Appl.} 33(1), 9--23 (1981)

\bibitem{Mord}
Mordukhovich, B.S.:
\newblock {\em Variational Analysis and Generalized Differentiation, vol. II}.
\newblock Springer-Verlag Berlin Heidelberg (2006)

\bibitem{ML2}
Moosaei, H., Bazikar, F., Ketabchi, S., Hlad\'{i}k, M.:
\newblock Universum parametric-margin $v$-support vector machine for classification using the difference of convex functions algorithm.
\newblock {\em Appl. Intell.} 52(3), 2634--2654 (2022)

\bibitem{Murty}
Murty, K.G., Kabadi, S.N.:
\newblock Some NP-complete problems in quadratic and nonlinear programming.
\newblock {\em Math. Program.} 39(2), 117--129 (1987)

\bibitem{Nam2018} Nam, {N.M.}, Geremew, W., Reynolds, R.,  Tran, T.:
Nesterov's smoothing technique and minimizing
differences of convex functions for hierarchical
clustering.
\newblock {\em  Optim. Lett.} 12(3), 455--473 (2018)

\bibitem{Nie} Nie, J., Yang, Z., Zhang, X.:
\newblock A complete semidefinite algorithm for detecting copositive matrices and tensors.
\newblock {\em SIAM J. Optim.} 28(4), 2902-2921 (2018)


\bibitem{Oliv10}
de Oliveira, W.:
\newblock Proximal bundle methods for nonsmooth DC programming.
\newblock {\em J. Global Optim.} 75(2), 523--563  (2019)

\bibitem{OT19}
de Oliveira, W., Tcheou, M.P.:
\newblock An inertial algorithm for DC programming.
\newblock {\em Set-Valued Var. Anal.} 27(4), 895--919 (2019)

\bibitem{TT97}
Pham Dinh, T., Le~Thi, H.A.:
\newblock Convex analysis approach to DC programming: theory, algorithms and applications.
\newblock {\em Acta Math. Vietnam.} 22(1), 289--355 (1997)

\bibitem{TT98}
Pham Dinh, T., Le~Thi, H.A.:
\newblock A D.C. optimization algorithm for solving the trust-region subproblem.
\newblock {\em SIAM J. Optim.} 8(2), 476--505 (1998)

\bibitem{TaoOMS}
Pham Dinh, T., Le~Thi, H.A., Akoa, F. :
\newblock Combining DCA (DC Algorithms) and interior point techniques for large-scale nonconvex quadratic programming.
\newblock {\em Optim. Methods Softw.} 23(4), 609--629 (2008)

\bibitem{Tuan2012}
Tuan {H.N.}:
\newblock Convergence rate of the Pham Dinh-Le Thi algorithm for the trust-region subproblem.
\newblock {\em J. Optim. Theory Appl.} 154(3), 904--915 (2012)

\bibitem{TuanJMAA}
Tuan {H.N.}:
\newblock Linear convergence of a type of iterative sequences in nonconvex quadratic programming.
\newblock {\em J. Math. Anal. Appl.} 423(2), 1311--1319 (2015)

\bibitem{Tuan2013}
Tuan {H.N.}, Yen, N.D.:
\newblock Convergence of the Pham Dinh-Le Thi's algorithm for the trust-region subproblem.
\newblock {\em J. Glob. Optim.} 55(2), 337-347 (2013)


\bibitem{RT}
Rockafellar, R.T.:
\newblock {\em Convex Analysis}.
\newblock Princeton University Press (1972)

\bibitem{RW}
Rockafellar, R.T., Wets, R.J.-B.:
\newblock {\em Variational Analysis},
\newblock Grundlehren Math. Wiss. 317, Springer, New York (1998)



\bibitem{XuXue2017}
Xu, H.M., Xue, H., Chen, {X.H.}, Wang, {Y.Y.}:
\newblock Solving indefinite kernel support vector machine with difference of convex functions programming.
\newblock {\em Proceedings of the Thirty-First AAAI Conference on Artificial Intelligence} (2017)


\end{thebibliography}

\end{document}